\documentclass[11pt]{article}
\usepackage{amssymb,amsfonts,amsmath,amsthm}
\usepackage{graphicx}
\usepackage{amscd}
\usepackage{graphicx}
\usepackage{microtype}
\usepackage{enumerate}
\usepackage{fullpage}
\usepackage{pinlabel}
\usepackage{xypic}
\usepackage[top=1 in,bottom=1.3in,left=1in,right=1in]{geometry}
\usepackage{pbox}
\usepackage{caption}
\usepackage{subcaption}

\theoremstyle{definition}
\newtheorem{theorem}{Theorem}[section]

\newtheorem{algorithm}[theorem]{Algorithm}

\newtheorem{conjecture}[theorem]{Conjecture}
\newtheorem{convention}[theorem]{Convention}
\newtheorem{construction}[theorem]{Construction}
\newtheorem{corollary}[theorem]{Corollary}
\newtheorem{definition}[theorem]{Definition}
\newtheorem{example}[theorem]{Example}
\newtheorem{exercise}[theorem]{Exercise}
\newtheorem{lemma}[theorem]{Lemma}

\newtheorem{problem}[theorem]{Problem}
\newtheorem{proposition}[theorem]{Proposition}
\newtheorem{remark}[theorem]{Remark}
\newtheorem{question}[theorem]{Question}
\newtheorem{basic problem}[theorem]{Basic Problem}
\newtheorem{rigdef}[theorem]{Rigidity Definition}

\newcommand{\R}{\mathbb{R}}
\newcommand{\Z}{\mathbb{Z}}

\newcommand{\Q}{\mathbb{Q}}
\newcommand{\C}{\mathbb{C}}
\newcommand{\F}{\mathcal{F}}
\newcommand{\del}{\partial}
\renewcommand{\H}{\mathbb{H}}

\newcommand{\boldhead}[1]{%
 {\medskip \noindent \bfseries #1  }}

\DeclareMathOperator{\isom}{Isom}
\DeclareMathOperator{\Hom}{Hom}
\DeclareMathOperator{\rot}{rot}
\DeclareMathOperator{\PSL}{PSL}
\DeclareMathOperator{\PSLk}{PSL^{(k)}}
\DeclareMathOperator{\SO}{SO}
\DeclareMathOperator{\SL}{SL}
\DeclareMathOperator{\SU}{SU}

\DeclareMathOperator{\rott}{\tilde{r}ot}

\DeclareMathOperator{\Homeo}{Homeo}
\DeclareMathOperator{\Diff}{Diff}

\DeclareMathOperator{\id}{id}
\DeclareMathOperator{\euler}{e}

\DeclareMathOperator{\fix}{fix}

\DeclareMathOperator{\Out}{Out}
\title{Rigidity and flexibility of group actions on the circle}
\author{Kathryn Mann}
\date{}

\begin{document}

\maketitle

\vspace{-.8cm}
\abstract{ We survey rigidity results for groups acting on the circle in various settings, from local to global and $C^0$ to smooth.   Our primary focus is on actions of surface groups, with the aim of introducing the reader to recent developments and new tools to study groups acting by homeomorphisms.}
\thispagestyle{empty}
\tableofcontents

\section{Introduction}

Given a group $\Gamma$ and a manifold $M$, can one describe all actions of $\Gamma$ on $M$?  Put in such broad terms this question is too ambitious, but certain special cases are quite tractable.  In this paper we focus on the special case $M=S^1$, and eventually further specialize to the case where $\Gamma$ is the fundamental group of a closed surface.   As should become clear, even this very special case leads to a remarkably rich theory, with deep connections to problems in topology, geometry and dynamics.  

Perhaps the reader is already familiar with the wonderful survey paper of Ghys, titled \emph{Groups acting on the circle} \cite{Ghys Ens}.  If not, we recommend it highly both on its own and as a companion to this work.   Although this survey also treats groups acting on $S^1$, we have chosen to take a rather different approach -- while Ghys starts by describing the dynamics of a single homeomorphism, and ultimately aims to prove that higher rank lattices do not act on on the circle, here we focus from the beginning on group actions, on groups that \emph{do} act on $S^1$, and on groups that act in many different ways.  Our aim is to explore and understand the rigidity and flexibility of these actions.    

\boldhead{Outline and scope.} We begin by motivating the study of \emph{spaces of group actions}, as well as our focus on surface groups.   Section \ref{intro rig sec} introduces the concept of rigidity in various forms, and Section \ref{themes sec} presents some well-known examples.  The remainder of the paper is devoted to advertising ``rotation number coordinates" as a means of studying spaces of group actions on $S^1$, culminating in a description of recent results and new approaches to old problems, from \cite{CW}, \cite{Invent}, and others.  

It is our hope that this paper will be both valuable and accessible to a wide audience.  Whether you are a graduate student looking for an introduction to group actions, an expert with an interest in character varieties and geometric structures, or instead have a background in hyperbolic dynamics, there should be something here for you.   

Regretfully, there are many topics we have been forced to omit or gloss over in the interest of brevity and accessibility.  These include regularity of actions, circular orders on groups, and bounded cohomology -- in fact, we have made the perhaps unconventional choice to keep discussions of group cohomology separate and in the background.   Of course, in all cases we have given references wherever possible, and hope this paper serves as a welcoming entry point for the interested reader.

\boldhead{Acknowledgements.}  Thanks to all who have shared their perspectives on (surface) groups acting on the circle, especially Christian Bonatti, Danny Calegari, Benson Farb, \'Etienne Ghys, Shigenori Matsumoto, Andr\'es Navas, Maxime Wolff, and the members of the Spring 2015 MSRI program \textit{Dynamics on Moduli Spaces of Geometric Structures}.  Thanks additionally to Subhadip Chowdhury, H\'el\`ene Eynard-Bontemps, Bena Tshishiku, and Alden Walker for comments on the manuscript, and to
the organizers and participants of \textit{Beyond Uniform Hyperbolicity 2015} where a portion of this work was presented as a mini-course.

\section{The ubiquity of surface groups}  \label{ubiq sec}

We begin by introducing a few of the many ways that actions of finitely generated groups come up in areas of topology, geometry and dynamics; with an emphasis on surface groups acting on $S^1$.  We'll see actions of surface groups on $S^1$ arise as the most basic case of \emph{flat} or \emph{foliated bundles}, as the essential examples of \emph{geometric actions on the circle}, and through analogy with the study of \emph{character varieties}.  

In addition to 
illustrating the diversity of questions 
related to surface groups acting on the circle, this section also covers basic background material used later in the text.

\begin{convention} ``Surface group" means the fundamental group of a closed, orientable surface of genus $g \geq 2$.
\end{convention}

\begin{convention} For convenience, we assume throughout this work that everything is oriented:  all manifolds are orientable, bundles are oriented, foliations are co-oriented, and homeomorphisms preserve orientation.  We use $\Diff^r(M)$ to denote the group of $C^r$ \emph{orientation-preserving} diffeomorphisms of a manifold $M$, and $\Homeo(M) = \Diff^0(M)$ the group of orientation-preserving homeomorphisms.  
\end{convention}

\subsection{Flat bundles and foliations}  \label{bundle subsec}

Let $M$ and $F$ be manifolds of dimension $m$ and $n$ respectively, and $E$ a smooth $F$--bundle over $M$.  The bundle $E$ is called \emph{flat} if any of the following equivalent conditions hold: 
\begin{enumerate}[i)]
\item $E$ admits a connection form with vanishing curvature. 
\item $E$ admits a smooth foliation of codimension $n$ transverse to the fibers.
\item $E$ admits a trivialization with locally constant transition maps (i.e. totally disconnected structure group)
\end{enumerate}

\noindent The foliation of condition ii) is given by the integral submanifolds of the (completely integrable) horizontal distribution given by the connection form, and is defined locally by the leaves $U \times \{p\}$ of the trivialization $U \times F \to U \subset M$ from condition iii).  

Viewing the foliation or trivialization as the central object, rather than the connection form, allows one to extend the notion of ``flat" to bundles of lower regularity.  

\begin{definition}
A (topological) \emph{flat $F$ bundle over $M$} is a topological $F$--bundle with a foliation of codimension $n$ transverse to the fibers.   
\end{definition}

Note that this definition requires only $C^0$ regularity -- to be transverse to the fibers is a statement about the existence of local charts to $\R^m \times \R^n$, with fibers mapped into to sets of the form $\R^n \times \{y\}$ and leaves of the foliation to $\{x\} \times \R^m$. 
   In this case, we define the \emph{regularity} of the flat bundle to be the regularity of the foliation.   

\boldhead{Holonomy and group actions.}  The relationship between flat bundles and groups acting on manifolds is given by the holonomy representation, which we describe now.  
Fix a basepoint $x \in M$, and an identification of $F$ with the fiber over $x$.  
The \emph{holonomy} of a flat bundle of class $C^r$ is a homomorphism 
$$\rho: \pi_1(M, x) \to \Diff^r(F)$$
defined as follows:  Given $p \in F$, and a loop $\gamma: [0,1] \to M$ based at $x$ and representing an element of $\pi_1(M, x)$, there is a unique lift of $\gamma$ to a horizontal curve $\tilde{\gamma}: [0,1] \to E$ with $\tilde{\gamma}(0) = p$ and $\tilde{\gamma}(1) \in F$.  Define $\rho(\gamma)(p) :=  \tilde{\gamma}(1)$.   One checks easily that $\rho(\gamma)(p)$ depends only on the based homotopy class of $\gamma$, that $\rho(\gamma)$ is a homeomorphism of $F$, and $\rho$ is a homomorphism $\pi_1(M, x) \to \Diff^r(F)$.  

Conversely, given a homomorphism $\rho: \pi_1(M, x) \to \Diff^r(F)$, one can build a flat bundle with holonomy $\rho$ as a quotient of the trivial $F$--bundle over the universal cover $\tilde{M}$.  The quotient is  
$$(\tilde{M} \times F) / \pi_1(M)$$
where $\pi_1(M)$ acts diagonally on $\tilde{M} \times F$ by deck transformations on $\tilde{M}$ and by $\rho$ on $F$; the foliation descends from the natural foliation of $\tilde{M} \times F$ by leaves of the form $\tilde{M} \times \{p\}$.  

These constructions produce a one-to-one correspondence
$$ \left\{ \text{Flat }F\text{--bundles of class }C^r \right\}    \overset{\text{holonomy}}\longleftrightarrow   \left\{ \text{representations } \pi_1(M, x) \to \Diff^r(F) \right\}.$$
This is more often stated as a correspondence between flat bundles \emph{up to equivalence} and representations \emph{up to conjugacy in $\Diff^r(F)$}.   

 Later, in the special case where $F = S^1$, we will also consider the equivalence relation on flat bundles that corresponds to \emph{semi-conjugate} representations.  Very roughly speaking, semi-conjugate flat bundles are obtained by replacing a leaf $L$ with a foliated copy of $L \times I$ (or replacing many leaves by foliated product regions), and/or by the inverse operation of collapsing foliated regions down to single leaves.  

From this perspective, questions about \emph{rigidity} and \emph{flexibility} of representations are questions about the (local) structure of the space of horizontal foliations on a given fiber bundle.    We'll begin to discuss these questions in Section \ref{intro rig sec}.  For now,  we focus on an even more fundamental problem:  

\begin{basic problem} \label{flat prob}
Which bundles admit a foliation transverse to the fibers?   Equivalently, which bundles can be obtained from a homomorphism $\rho: \pi_1(M) \to \Diff^r(F)$ as described above? 
\end{basic problem}

Despite its seemingly simple statement, the only nontrivial case in which Problem \ref{flat prob} has a complete answer is when the fiber $F$ is $S^1$ and base $M$ is a surface.  This answer is given by the \emph{Milnor--Wood inequality}.

\boldhead{Characterizing flat bundles: The Milnor--Wood inequality.}
\begin{convention}
For the rest of the text, $\Sigma_g$ denotes the closed surface of genus $g$, assuming always that $g \geq 2$.  
\end{convention}

The Milnor--Wood inequality characterizes which circle bundles over $\Sigma_g$ admit a foliation transverse to the fibers in terms of a classical invariant called the \emph{Euler number}.

\begin{theorem}[Milnor--Wood inequality, \cite{Milnor} \cite{Wood}]  \label{MW thm}
A topological $S^1$ bundle over $\Sigma_g$ admits a foliation transverse to the fibers if and only if the \emph{Euler number} $\euler$ of the bundle is less than or equal to $-\chi(\Sigma_g)$ in absolute value.    
\end{theorem}

We will give an elementary definition of the Euler number in Section \ref{euler subsec} and eventually give a proof of the Milnor--Wood inequality as well\footnote{For impatient readers who happen to be familiar with classifying spaces, here is a quick definition of the Euler number to tide you over.
\begin{definition} \label{eu def 1}
Let $E$ be a topological $S^1$ bundle over a surface $\Sigma$.   There is an isomorphism $H^2(B\Homeo(S^1); \Z) \cong
 H^2(B\SO(2); \Z) \cong \Z$ induced by the inclusion of $\SO(2)$ into $\Homeo(S^1)$. Thus, the Euler class for $SO(2)$--bundles in $H^2(B\SO(2); \Z)$ (a generator) pulls back to a generator of $H^2(B\Homeo(S^1); \Z)$.  
The \emph{(integer) Euler class} of $E$ in $H^2(\Sigma; \Z)$ is the pullback of this element under the classifying map.  The \emph{Euler number} is the integer obtained by evaluating the Euler class on the fundamental class of $\Sigma$.
\end{definition}
}.

\medskip

As for bundles over surfaces with fibers \emph{other} than $S^1$,  the question of which bundles admit horizontal foliations is quite difficult, even when the fiber is a surface.  As an example, Kotschick and Morita \cite{KM} recently gave 
examples of flat surface bundles over surfaces with nonzero signature, but whether various other characteristic classes can be nonzero on flat bundles remains open.

\subsection{Geometric actions on manifolds}

Our second motivation for studying surface group actions on the circle comes from the notion of a geometric action on a manifold. Recall that a \emph{cocompact lattice} $\Lambda$ in a Lie group $G$ is a discrete subgroup such that $G/\Lambda$ is compact.  

\begin{definition} \label{geom def}
Let $\Gamma$ be a finitely generated group, and $M$ a manifold. A representation $\rho: \Gamma \to \Homeo(M)$ is called \emph{geometric} if it is faithful and with image a cocompact lattice in a transitive, connected Lie group $G \subset \Homeo(M)$.  
\end{definition} 

This definition is motivated by Klein's notion of a \emph{geometry} as a pair $(X, G)$, where $X$ is a manifold and $G$ a Lie group acting transitively on $X$.  In the case where $M$ is a manifold of the same dimension as $X$, a \emph{geometric structure} or \emph{$(G, X)$ structure} on $M$, in the sense of Klein, is specified by a representation $\pi_1(M) \to G \subset \Homeo(X)$.
As a basic example, any compact surface of genus $g \geq 2$ admits a hyperbolic structure -- in fact, many different hyperbolic structures -- and these correspond roughly\footnote{Technically, it is \emph{marked} hyperbolic structures that correspond to conjugacy classes of discrete, faithful representations.} to the conjugacy classes of discrete, faithful representations $\rho: \pi_1(\Sigma_g) \to \PSL(2, \R) = \isom(\H^2)$.  The correspondence comes from considering $\Sigma_g = \H^2/\rho(\Gamma_g)$ with the induced hyperbolic metric. 
But this is a bit of a digression, at the moment we are more interested in the action of $\PSL(2,\R)$ on the circle...

\boldhead{Geometric actions on $S^1$.}
Specializing to the situation $M = S^1$, it is not hard to give a complete classification of the connected Lie groups acting transitively on $M$.  These are $\SO(2)$, the group of rotations; $\PSL(2,\R)$, the group of M\"obius transformations; and the central extensions $\PSLk$ of $\PSL(2, \R)$ of the form
$$0 \to \Z/k\Z \to \PSLk \to \PSL(2,\R) \to 1.$$
The natural action of $\PSLk$ on $S^1$ comes from taking all lifts of M\"obius transformations of $S^1$ to the $k$-fold cover of $S^1$ (which is, conveniently, also a circle).  As a concrete example, $\PSL^{(2)} = \SL(2, \R)$ with its standard action on $S^1$ considered as the space of rays from the origin in $\R^2$.  
A nice proof of this classification of Lie subgroups can be found in \cite[Sect. 4.1]{Ghys Ens}; the essential observation -- originally due to Lie -- is that a Lie group acting faithfully on a 1-manifold can have at most a 3-dimensional Lie algebra.  
 
Given this classification, it is only a small step to describe all geometric actions.

\begin{theorem}[Geometric actions on $S^1$]  
Let $\rho: \Gamma \to \Homeo(S^1)$ be a geometric action.  Either
\begin{enumerate}[i)]
\item $\Gamma$ is finite cyclic, and $\rho$ is conjugate to a representation into $\SO(2)$, or
\item $\Gamma$ is a finite extension of a surface group, and $\rho$ is conjugate to a representation with image in $\PSLk$.  
\end{enumerate}
\end{theorem}
\noindent In other words, up to finite index, all geometric actions on $S^1$ come from surface groups.  

\begin{proof}
The discrete, cocompact subgroups of $\SO(2)$ are exactly the finite cyclic groups.   Suppose now instead that $H$ is a cocompact subgroup of $\PSLk$.  The image of $H$ under projection to $\PSL(2,\R)$ is a discrete, cocompact subgroup of $\PSL(2,\R)$, 
so contains a surface group as finite index subgroup.  It follows that $H$ is a subgroup of the central extension of this surface group by $\Z/k\Z$ and hence itself a finite extension of a surface group.    
\end{proof}

There are many examples of geometric actions of surface groups on the circle.  The general construction is as follows.  

\begin{example}[Surface groups in $\PSLk$] \label{pslk ex}
Let $\Gamma_g = \pi_1(\Sigma_g)$ with standard presentation $\langle a_1, b_1, ... a_g, b_g \mid \prod_i [a_i, b_i] \rangle$.  Let  $\rho: \Gamma_g \to \PSL(2,\R)$ be a discrete, faithful representation.  For each generator $a_i$ and $b_i$ of $\Gamma_g$, pick a lift $\tilde{\rho}(a_i)$ of $\rho(a_i)$ and $\tilde{\rho}(b_i)$ of $\rho(b_i)$ to $\PSLk$.   One can check directly that the relation $\prod_i [\tilde{\rho}(a_i), \tilde{\rho}(b_i)] = \id $ is satisfied precisely when $k$ divides $2g-2$; in this case, we have defined an action of $\Gamma_g$ on the ($k$-fold covering) circle by homeomorphisms in $\PSLk$.  
\end{example}

\begin{remark}[for the experts...]  
Example \ref{pslk ex} can be nicely rephrased in more cohomological language: the obstruction to lifting an action of $\Gamma_g$ on $S^1$ to the $k$-fold cover is exactly \emph{divisibility of the Euler class by k}, as the lifted action gives a flat bundle which is the $k$-fold fiber-wise cover of the original. Faithful representations of $\Gamma_g$ with image a lattice in $\PSL(2,\R)$ correspond to bundles topologically equivalent to the unit tangent bundle of $\Sigma_g$, which has Euler number $\pm \chi(\Sigma)$, depending on orientation.  
\end{remark}

We'll see that geometric actions are one source of rigidity, and in the case of actions of surface groups on $S^1$, conjecturally the \emph{only} source of a strong form of $C^0$ rigidity (see Theorem \ref{invent thm} and Conjecture \ref{invent conj} below).

\subsection{Representation spaces and character varieties}  \label{char var subsec}

Closely related to the theme of geometric actions is the classical study of representation spaces and character varieties.  This subject also has particularly strong ties to surface groups.  

\begin{definition}
Let $\Gamma$ be a discrete group, and $G$ a topological group.  The \emph{representation space} $\Hom(\Gamma, G)$ is the set of homomorphisms $\rho: \Gamma \to G$, with the topology of pointwise convergence. 

\end{definition}
\noindent In this topology, $\rho_n$ converges to $\rho$ if $\rho_n(\gamma)$ converges to $\rho(\gamma)$ in $G$ for each $\gamma \in \Gamma$.  This is equivalent to requiring convergence of $\rho_n(s)$ for each $s$ in a generating set for $\Gamma$.  It is sometimes also called the ``algebraic" topology on $\Hom(\Gamma, G)$.  
\medskip

The \emph{character variety} $X(\Gamma, G)$ is, roughly speaking, the quotient of $\Hom(\Gamma, G)$ by the conjugation action of $G$.  
To explain this terminology (why ``variety"?  why only ``roughly"?), we detour for a minute into the special case where $G$ is an linear algebraic group.  
In this case, $\Hom(\Gamma, G)$ has the structure of an algebraic variety whenever $\Gamma$ is finitely presented.  Concretely, if $S$ is a finite generating set for $\Gamma$, then $\Hom(\Gamma, G)$ can be realized as a subset of $G^{|S|}$ cut out by finitely many polynomial equations coming from the relators in $\Gamma$.   Unfortunately, the quotient $\Hom(\Gamma, G)/G$ is not a variety -- in general the quotient is not even Hausdorff.  However, there is a better way to take a kind of algebraic-geometric quotient (the details of which we will not discuss here) and the resulting (Hausdorff, nice) space is called the \emph{character variety}. In many cases, it can genuinely be identified with the variety of characters of representations $\Gamma \to G$.   See e.g. \cite{Shalen} for a thoughtful introduction to the special case $G = \PSL(2,\C)$, and \cite{Labourie} for a more general treatment.

\boldhead{Character varieties of surface groups.}
There are a number reasons why character varieties of surface groups are particularly interesting, many coming from the perspective of moduli spaces of geometric structures.   For example, we mentioned above the relationship between discrete, faithful representations $\Gamma_g \to \PSL(2,\R)$ and hyperbolic structures on $\Sigma_g$.  There are analogous relationships between convex real projective structures and certain representations $\Gamma_g \to \mathrm{PGL}(3,\R)$, between Higgs bundles on Riemann surfaces and representations to $\SL(n, \C)$, etc.\footnote{See \cite{Goldman survey} for a more detailed description of the relationships between moduli spaces and character varieties}    
Representation spaces of surface groups are also interesting in their own right -- for example, they admit a natural symplectic structure (discovered by Atiyah and Bott for spaces of representations into compact Lie groups, and Goldman in the noncompact case).  We cannot possibly do justice to the theory here, and instead refer the reader to the recent book of Labourie \cite{Labourie}.  

However, it will be convenient for us to say a few words about the basic motivating case $\Hom(\Gamma_g, \PSL(2,\R))$.  For the sake of simplicity, we will discuss only representation spaces and not character varieties.

\boldhead{The structure of $\Hom(\Gamma_g, \PSL(2,\R))$.}
As $\PSL(2,\R)$ is an algebraic group, and $\Gamma_g$ finitely presented, $\Hom(\Gamma_g, \PSL(2,\R))$ is a real algebraic variety. In particular, this implies that $\Hom(\Gamma_g, \PSL(2,\R))$ has only finitely many components.  

How many?  The Euler number of representations can be used to give a lower bound.  By the Milnor--Wood inequality (Theorem \ref{MW thm}), there are $4g-3$ possible values of the Euler number of a representation $\rho \in \Hom(\Gamma_g, \PSL(2,\R))$,  ranging from $\chi(\Sigma_g)$ to $-\chi(\Sigma_g)$.  It's not hard to construct examples to show that each value in this range is attained:  in brief, a hyperbolic structure on $\Gamma_g$ defines a representation with Euler number $-2g+2$, a singular hyperbolic structure with two order 2 cone points has Euler number $-2g+1$, and every other negative value can then be constructed by a surjection $\Gamma_g \to \Gamma_h$ followed by a representation $\Gamma_h \to \PSL(2, \R)$ as above.  Positive values can then be attained by conjugating by an orientation-reversing map of $S^1$.  

We'll also see later that the Euler number varies continuously, and hence is constant on connected components of $\Hom(\Gamma_g, \PSL(2,\R))$.  This means that $\Hom(\Gamma_g, \PSL(2,\R))$ has at least $4g-3$ connected components.    A remarkable theorem of Goldman states that this easy bound is sharp.

\begin{theorem}[Goldman \cite{Goldman}] \label{goldman thm}
$\Hom(\Gamma_g, \PSL(2,\R))$ has $4g-3$ connected components, completely classified by the Euler number.
\end{theorem}

The situation is quite different when $G$ is not algebraic.  For example,

\begin{question}[\textbf{Open}] \label{comp quest}
Does $\Hom(\Gamma_g, \Homeo(S^1))$ have finitely many connected components?  What about $\Hom(\Gamma_g, \Diff^r(S^1))$, for any $r$?  
\end{question}

Only very recently was it shown that the Euler number does \emph{not} classify connected components of $\Hom(\Gamma_g, \Homeo(S^1))$.  A particular consequence of the work in \cite{Invent} and \cite{Bowden} gives the following.  

\begin{theorem}[\cite{Invent}; \cite{Bowden} for the case $G = \Diff^\infty(S^1)$] \label{comp thm}
Let $G = \Diff^r(S^1)$ for any $0 \leq r \leq \infty$.  There are more than $2^g$ connected components of $\Hom(\Gamma_g, G)$ consisting of representations with Euler number $g-1$.  
\end{theorem}

The essential ingredient in Theorem \ref{comp thm} is a \emph{rigidity result} (Theorem \ref{invent thm} below), which is one of the primary motivations of this paper.  Theorem \ref{invent thm} in fact says much more about the number of connected components of $\Hom(\Gamma_g, G)$ -- see Theorem \ref{invent comp cor} below.  
We'll explain the theorem and its proof in section \ref{pf sec}.  Unfortunately, the proof does not seem to give a hint at whether Question \ref{comp quest} has a positive or negative answer!

\begin{remark}[Regularity matters]
Although the statement of Theorem \ref{comp thm} is uniform over all $G= \Diff^r(S^1)$, i.e. independent of $r$, one generally expects the topology of $\Hom(\Gamma_g, \Diff^r(S^1))$ to depend on $r$.  For instance, for the \emph{genus 1} surface group $\Z^2$, we have
\begin{itemize}
\item Easy theorem: $\Hom(\Z^2, \Homeo(S^1))$ is connected.
\item Highly nontrivial theorem: $\Hom(\Z^2, \Diff^1(S^1))$ is path connected (Navas, \cite{Navas comp}).
\item Open question: is $\Hom(\Z^2, \Diff^1(S^1))$ \emph{locally} (path) connected?
\item Open question: is $\Hom(\Z^2, \Diff^r(S^1))$ connected, for any $r \geq 2$? 
\end{itemize}
See \cite{BE} for discussion and progress on the $\Hom(\Z^2, \Diff^\infty(S^1))$ case, and \cite{Navas} for many discussions of (often quite subtle) issues of regularity of group actions on the circle.  
\end{remark}

\boldhead{Infinite dimensional Lie groups?}
Goldman's proof of Theorem \ref{goldman thm} relies heavily on the Lie group (i.e. manifold) structure of $\PSL(2,\R)$.   It is conceivable, though perhaps optimistic, that some of his strategy might carry over to the study of representations to $\Diff^r(S^1)$.   For any manifold $M$,  and $r \geq 1$, the group $\Diff^r(M)$ has a natural smooth structure (as a Banach manifold, or Fr\'echet manifold for the case of $\Diff^\infty(M))$, and the tangent space at the identity can be identified with the Lie algebra of $C^r$ vector fields on $M$, making $\Diff^r(M)$ an infinite dimensional Lie group.   In the case $M=S^1$, the structure of the Lie algebra $\mathrm{Vect}(S^1)$ is quite well understood. 
Although this doesn't carry over for $\Homeo(M)$, there are good candidates for the ``Lie algebra" of homeomorphisms of the circle, see e.g \cite{MP}.

\section{An introduction to rigidity} \label{intro rig sec}

We turn now to the main theme of this work, \emph{rigidity} of group actions on $S^1$.  We begin by introducing several notions of rigidity, ranging from local to global.   

\subsection{Local rigidity}
Loosely speaking, an action of a group $\Gamma$ is \emph{rigid} if deformations of this action are ``trivial" in some sense, usually taken to mean that they are all conjugate.   There are several ways to formalize this notion, for example:

\begin{rigdef}[Local rigidity]  \label{loc rig}
Let $G$ be a topological group.  A representation $\rho: \Gamma \to G$ is \emph{locally rigid} if there is a neighborhood $U$ of $\rho$ in $\Hom(\Gamma, G)$ such that each $\rho' \in U$ is conjugate to $\rho$ in $G$.
\footnote{Often it is additionally required that the conjugacy be by an element of $G$ close to the identity.}
\end{rigdef}

In the special case $G = \Diff^\infty(M)$, there is a stronger version of local rigidity called \emph{differentiable rigidity}.  Here, the conjugacy has higher regularity than expected.  

\begin{rigdef}[Local differentiable rigidity, classical sense]  \label{rigdef2}
A representation $\rho: \Gamma \to \Diff^\infty(M)$ has \emph{local differentiable rigidity} if every representation in $\Hom(\Gamma, \Diff^\infty(M))$ sufficiently close to $\rho$ in the \emph{$C^1$ topology} on $\Diff^\infty(M)$ is (smoothly) conjugate to $\rho$. 
\end{rigdef}

\noindent There are many possible variants of Definition \ref{rigdef2} -- one can replace $C^1$ and smooth with $C^r$ and $C^s$, one can ask only that representations in $\Hom(\Gamma, \Diff^\infty(M))$ close to $\rho$ \emph{and} $C^0$ conjugate to $\rho$ are smoothly conjugate to $\rho$, etc.  

\medskip
In the opposite direction, \emph{structural stability} is a notion of rigidity where the conjugacy is less regular than the perturbation.  
Classically, a flow $f_t$ on manifold $M$ is said to be \emph{structurally stable} if every sufficiently small $C^1$ perturbation $g_t$ of the flow is topologically equivalent to $f_t$, meaning that there is a homeomorphism of $M$ mapping flowlines of $g_t$ to flowlines of $f_t$.  A natural adaptation of this definition to group actions is the following. 

\begin{rigdef}[Structural stability]  \label{struc def}
A representation $\rho: \Gamma \to \Diff^1(M)$ is \emph{structurally stable} if there is a neighborhood $U$ of $\rho$ in $\Hom(\Gamma, \Diff^1(M))$ consisting of representations $C^0$ conjugate to $\rho$.  
\end{rigdef}

In the next section we'll give two proofs of structural stability for actions of surface groups on the circle.

\boldhead{Deformations in $\Homeo(S^1)$.}
In the case where $G = \Homeo(S^1)$, there is another reasonable notion of ``trivial deformation" of an action.  This comes from a construction of Denjoy.
\begin{construction}[``Denjoy-ing" an action] \label{denjoy trick}
Let $\Gamma$ be a countable group, and $\rho \in \Hom(\Gamma, \Homeo(S^1))$.  Enumerate the points of an orbit $\rho(\Gamma) x = \{x_1, x_2, x_3, ... \}$, and modify $S^1$ by replacing each point $x_i$ with an interval of length $t 2^{-i}$ to get a new circle $S_t$.   There is a unique affine map from the interval inserted at a point $x_i$ to the interval inserted at any $\rho(\gamma) x_i = x_j$, this gives a well-defined means of extending $\rho$ to an action $\rho_t$ of $\Gamma$ on the circle $S_t$.   As $t \to 0$, this action approaches the original action $\rho$ in $\Hom(\Gamma, \Homeo(S^1))$, provided that one makes a reasonable identification between $S^1$ and the larger circles $S_t$.  
\end{construction}

Generalizing Denjoy's construction is the notion of \emph{semi-conjugacy}, which we already hinted at in our discussion of equivalence of flat bundles in Section \ref{bundle subsec}.  Recall that two representations $\rho$ and $\rho': \Gamma \to \Homeo(S^1)$ are \emph{conjugate} if there is a homeomorphism $h$ of $S^1$ such that 
$$h \rho(\gamma)(x) = \rho'(\gamma) h(x), \text{ for all } \gamma \in \Gamma \text{ and } x \in S^1.$$ 
Semi-conjugacy is essentially the same notion, but $h$ is no longer required to be a homeomorphism: it is permitted to collapse intervals to points, or have points of discontinuity where it performs the Denjoy trick, blowing up points into intervals.   To make the definition completely formal requires a bit of care, so we defer the work to Section \ref{rot subsec}. 
There, we will also see that semi-conjugate representations have many of the same dynamical properties, so it is entirely reasonable to consider a semi-conjugacy to be a ``trivial deformation". 

As construction \ref{denjoy trick} can be applied to any action of a countable group, there is no hope that such groups could be locally rigid in the sense of Definition \ref{loc rig}.  Instead, we make the following obvious modification.

\begin{rigdef}[Local rigidity in $\Homeo(S^1)$]
Let $\Gamma$ be a countable group.  A representation $\rho: \Gamma \to \Homeo(S^1)$ is \emph{locally rigid} if there is a neighborhood $U$ of $\rho$ in $\Hom(\Gamma, G)$ such that each $\rho' \in U$ is semi-conjugate to $\rho$.
\end{rigdef} 

Although Construction \ref{denjoy trick} produces an action by homeomorphisms, there is a way to construct semi-conjugate $C^1$-actions of certain groups by a similar trick (due to Denjoy and Pixton).  However, the construction cannot generally be made $C^2$, and the case of intermediate regularity is quite interesting!  A detailed discussion of such regularity issues is given by Navas in \cite[Chapt. 3]{Navas}.

\subsection{Full rigidity}
If one considers not only small deformations of a representation $\rho: \Gamma \to G$, but \emph{arbitrary} deformations, one is led to a statement about the connected component of $\rho$ in $\Hom(\Gamma, G)$.   Rigidity then should mean that the connected component is ``as small as possible".   There is a particularly nice way to formalize this notion in the case where $G = \Homeo(S^1)$, we call this \emph{full rigidity}.  (Although ``global rigidity" or ``strong rigidity" would also make sense, these names were already taken!)   

\begin{rigdef}[Full rigidity] \label{totrig def}
A representation $\rho: \Gamma \to \Homeo(S^1)$ is \emph{fully rigid} if the connected component of $\Hom(\Gamma, \Homeo(S^1))$  containing $\rho$ consists of a single semi-conjugacy class.  
\end{rigdef}

\noindent The set of all representations semi-conjugate to a given one is path-connected (see the remark after Definition \ref{C-D def}), so this definition really does describe representations with the smallest possible connected component.    

An easy example of full rigidity is the following. We leave the proof as an exercise. 
\begin{example} \label{finite ex}
The inclusion of $\Z/m\Z$ into $\Homeo(S^1)$ as a group of rotations is fully rigid.  In fact, \emph{every} action of any finite group on $S^1$ is fully rigid.
\end{example}

\noindent A much less trivial example of fully rigid actions follows from a theorem of Matsumoto, Theorem \ref{mats thm}, which we will state in Section \ref{max subsec}.  

\begin{example}[Consequence of Theorem \ref{mats thm}]  \label{mats ex}
Let $\rho: \Gamma_g \to \Homeo(S^1)$ be a faithful representation with image a discrete subgroup of $\PSL(2,\R)$.  Then $\rho$ is fully rigid.  
\end{example}

In Section \ref{themes sec}, we'll also describe a  theorem on \emph{geometric actions} subsuming both of the above examples (c.f. Theorem \ref{invent thm} ``geometric actions on $S^1$ are fully rigid").

\subsection{Strong (Mostow) Rigidity} \label{mostow subsec}

Although not our primary focus here, we mention briefly a third definition of rigidity coming from the tradition of lattices in Lie groups.  A more in-depth introduction can be found in \cite{GP} or \cite{Spatzier}.  

\begin{rigdef}[Strong rigidity]
A lattice $\Gamma$ in a Lie group $G$ is \emph{strongly rigid} if any homomorphism $\rho: \Gamma \to G$ whose image is also a lattice extends to an isomorphism of $G$.  
\end{rigdef}

Mostow proved strong rigidity for cocompact lattices in $G = \isom(\H^n)$, provided that $n \geq 3$; this was extended by Prasad to lattices in isometry groups of other locally symmetric spaces.  Margulis' superrigidity theorem can be seen as a further generalization of this kind of result.  

But of interest to us is what happens to Mostow's original rigidity theorem when $n = 2$.   
The cocompact lattices in $\isom(\H^2) = \PSL(2,\R)$ are surface groups, and they are \emph{not} rigid in $\PSL(2,\R)$.  Rather, there is a $6g-6$-dimensional \emph{Teichm\"uller space} of $\PSL(2,\R)$-conjugacy classes of discrete, faithful representations $\Gamma_g \to \PSL(2,\R)$, forming two connected components of the character variety $\Hom(\Gamma_g, \PSL(2,\R))/ \PSL(2,\R)$.   As we mentioned earlier, these parametrize the (marked) hyperbolic structures on $\Sigma_g$.  

However, it is still possible to recover a kind of rigidity for these representations by thinking of $\PSL(2,\R)$ as a group of homeomorphisms of the circle.
Although not conjugate in $\PSL(2,\R)$, any two discrete, faithful representations $\Gamma_g \to \PSL(2,\R)$ \emph{are} conjugate by a  (possibly orientation-reversing) homeomorphism of $S^1$.  To see this using a little hyperbolic geometry think of $\PSL(2,\R)$ acting by M\"obius transformations on the Poincar\'e disc model of $\H^2$.  Two discrete faithful representations of $\Gamma_g$ will have homeomorphic fundamental domains, a homeomorphism between them can be extended equivariantly to a homeomorphism of the disc, including the boundary circle, which induces the conjugacy. 

We conclude this section by mentioning an even stronger rigidity result,  in the spirit of ``differentiable rigidity", that comes from considering the regularity of the conjugating homeomorphism.  

\begin{theorem}[``Mostow rigidity on the circle" \cite{Mostow}]   \label{mostow circle} 
Let $\rho$ and $\rho'$ be discrete, faithful representations of $\Gamma_g$ into $\PSL(2,\R)$; with $h \rho h^{-1} = \rho'$. 
If the derivative $h'(x)$ is nonzero at any point $x$ where it is defined, then $h \in \PSL(2,\R)$.  
\end{theorem}
\noindent Further discussion can be found in the survey paper \cite{Agard}.

\subsection{The Zimmer program} 
The strong rigidity theorems of Mostow, Margulis, and others place strict constraints on representations of lattices in Lie groups.  The \emph{Zimmer program} is a series of conjectures organized around the theme that actions of lattices on \emph{manifolds} should be similarly constrained.  (See Fisher's survey paper \cite{Fisher} for a nice introduction.)

The main successes of the Zimmer program, such as Zimmer's \emph{cocycle rigidity}, use ergodic theory and apply only to actions by measure-preserving homeomorphisms or diffeomorphisms.  However, there is also a small family of nice results regarding lattices acting on the circle.   We give two examples; the first a beautiful instance of the Zimmer program catchphrase ``large (rank) groups do not act on small (dimensional) manifolds", the second a classification of all actions of lattices on $S^1$.  

\begin{theorem}[Witte-Morris, \cite{Witte}]
Let $\Gamma$ be an arithmetic lattice in an algebraic semi-simple Lie group of $\Q$-rank at least 2 (for example, a finite index subgroup of $\SL(n, \Z)$, $n \geq 3$).  Then any representation $\rho: \Gamma \to \Homeo(S^1)$ has finite image.  
\end{theorem}

\begin{theorem}[Ghys, Theorem 3.1 in \cite{Ghys invent}]  \label{ghys inv thm}
Let $\Gamma$ be an irreducible lattice in a semi-simple Lie group of real rank at least 2. (A basic example is $G = \PSL(2,\R) \times \PSL(2,\R)$.)   \emph{Up to semi-conjugacy}, any action $\rho: \Gamma \to \Homeo(S^1)$ either has finite image or is obtained by a projection of $G$ onto a $\PSL(2,\R)$ factor: 
$$\Gamma \hookrightarrow G \twoheadrightarrow \PSL(2,\R) \subset \Homeo(S^1).$$
\end{theorem} 

Ghys' theorem does not imply Witte-Morris' result, as a representation $\rho$ that is semi-conjugate to one with finite image need only have a finite orbit, not finite image.  But the theorem does give new examples of fully rigid representations.  
\begin{corollary}
Let $\Gamma$ be an irreducible lattice in $\PSL(2,\R) \times G$, where $G$ is a semi-simple Lie group with no $\PSL(2,\R)$ factors.  Then the representation $\rho: \Gamma \twoheadrightarrow \PSL(2,\R) \subset \Homeo(S^1)$ is fully rigid in $\Hom(\Gamma, \Homeo(S^1))$.  
\end{corollary}

\boldhead{Thompson's group.}  
Thompson's group $G$ is the group of piecewise-affine homeomorphisms of the circle that preserve dyadic rational angles.  Brown and Geoghegan -- and others -- have suggested thinking of it as the analog of a ``higher rank lattice" subgroup for $\Homeo(S^1)$.  It is also known to be rigid, at least if one restricts to $C^2$ actions.  

\begin{theorem}[Ghys--Sergiescu \cite{GS}] 
All nontrivial homomorphisms $\phi: G \to \Diff^2(S^1)$ are semi-conjugate to the standard inclusion.  
\end{theorem} 

To the best of our knowledge, it is an open question whether this result extends to rigidity of homomorphisms $G \to \Homeo(S^1)$.

\section{Two perspectives on rigidity of group actions}  \label{themes sec}

Having seen some samples of rigidity of group actions, we now turn to a discussion of what phenomena (other than ``being a higher-rank lattice") can \emph{lead} to this rigidity.   We focus on two perspectives that have particularly strong ties to actions of surface groups on the circle.  

The first, \emph{hyperbolicity}, comes from classical smooth dynamics.   Although hyperbolicity is an essentially $C^1$ notion -- it is a statement about derivatives -- we also present a slightly different notion of hyperbolicity for group actions due to Sullivan that can be applied to actions by homeomorphisms.   As an illustration, we give two different proofs of structural stability for surface group actions, both with a hyperbolic flavor.  

Our second theme is \emph{geometricity}.  We motivate this with a discussion of \emph{maximal representations} of surface groups into Lie groups and generalizations to $\Homeo(S^1)$.    We argue that geometricity, rather than maximality, is the right way to think of the phenomenon underlying these rigidity results, and advertise rigidity of geometric representations as an organizing principle.

\subsection{Rigidity from hyperbolicity}   \label{hyp sec}

Broadly speaking, hyperbolic dynamics deals with differentiable dynamical systems that exhibit expanding and contracting behavior in tangent directions.  
The classic example of this is an \emph{Anosov diffeomorphism}. 

\begin{definition}[Anosov diffeomorphism]  \label{anosov def}
A diffeomorphism $f$ of a manifold $M$ is \emph{Anosov} if there exists a continuous $f$-invariant splitting of the tangent bundle 
$TM = E^u \oplus E^s$, and constants $\lambda > 1$,  $C > 0$ such that
\begin{align*}
& \| Df^n(v) \| \geq C \lambda^n \|v\| \text{ for } v \in E^u, \, n \geq 0 \text{ and} \\
& \| Df^{-n}(v) \| \geq C \lambda^n \|v\| \text{ for } v \in E^s, \,n \geq 0.  
 \end{align*}
\end{definition}

A slight variation on this is an \emph{Anosov flow}:  a flow $f_t$ on $M$ is Anosov if there is a continuous invariant splitting $TM = E^u \oplus E^s \oplus E^0$, where $E^u$ and $E^s$ are expanding and contracting as in Definition \ref{anosov def}, and $E^0$ is tangent to the flow.   

\boldhead{``Hyperbolicity implies structural stability"}
A major theme in hyperbolic dynamics is extracting rigidity from hyperbolic behavior.   A basic instance of this is the following theorem of Anosov.  

\begin{theorem}[\cite{Anosov}] \label{anosov thm1}
Anosov flows are structurally stable.  
\end{theorem}

Theorem \ref{anosov thm1} has been generalized to many kinds of diffeomorphisms and flows with weaker forms of hyperbolic behavior than the Anosov condition, such as \emph{partial hyperbolicity}.  This is still an active area of research -- the survey paper \cite{BPSW} provides a nice exposition of recent results in the partially hyperbolic case.  We also recommend \cite{Hasselblatt} for a general introduction to hyperbolic dynamics.  

Here we will focus on a very particular case of Anosov's original structural stability theorem, one that applies (of course!) to certain circle bundles over surfaces.  

\begin{theorem}[Anosov \cite{Anosov}] \label{anosov thm}
Let $\Sigma_g$ be a surface with hyperbolic metric, and $E$ the unit tangent bundle of $\Sigma_g$.  Then the geodesic flow on $E$ is structurally stable.  
\end{theorem}

The unit tangent bundle $E$ has the structure of a flat circle bundle over $\Sigma_g$, its holonomy is the representation $\Gamma_g \to \PSL(2,\R) \subset \Homeo(S^1)$ defining the hyperbolic structure on the surface.   What is more, the foliation transverse to the fibers also contains information about the geodesic flow: its tangent space is spanned by the flow direction $E^0$ together with the expanding direction $E^u$.   

As a consequence, one can derive a structural stability result for these representations.  

\begin{corollary}[Structural stability of group actions]  \label{anosov cor}
Let $\rho: \Gamma_g \to \Diff^\infty(S^1)$ be faithful with image a discrete subgroup of $\PSL(2,\R)$.  Then $\rho$ is structurally stable in the sense of Definition \ref{struc def}. 
\end{corollary}

The proof is not difficult conceptually, but it requires that you know a little more about Anosov flows than we have introduced here.  We give a quick sketch for the benefit of those readers who have the background.  More details can be found in \cite{Ghys AIF}.  

\begin{proof}[Proof sketch, following \cite{Ghys AIF}]
Let $f_t$ denote geodesic flow on the unit tangent bundle $E$ of $\Sigma_g$.  
Let $\rho': \Gamma_g \to \Diff^\infty(S^1)$ be a $C^1$ perturbation of $\rho$, and $E'$ the resulting flat bundle.  This is equivalent to $E$ as a smooth bundle, so we can think of it as the space $E$ with a foliation $\F'$ transverse to the fibers.  As $\rho'$ is $C^1$-close to $\rho$, the foliation $\F'$ is $C^1$ close to the transverse foliation for $\rho$, which is the \emph{weak unstable} foliation $E^0 \oplus E^u$ of $f_t$. 

As a consequence, the intersection of $\F'$ with the weak stable foliation of $f_t$ is a 1-dimensional foliation.  This foliation can be parametrized by a flow $g_t$ that is $C^1$-close to $f_t$.  Moreover, the weak unstable of $g_t$ is exactly $\F'$ -- although strong stable and unstable foliations are not invariant under reparametrization, the weak foliations are.  
Theorem \ref{anosov thm} now implies that $g_t$ and $f_t$ are topologically conjugate.  As such a topological conjugacy maps the weak unstable foliation of $f_t$ to that of $g_t$, 
it defines a conjugacy of $\rho$ and $\rho'$.
\end{proof}

\boldhead{Ghys' differentiable rigidity.}
Building on Corollary \ref{anosov cor}, Ghys proved a remarkable differentiable rigidity result for surface group actions.  

\begin{theorem}[Local differentiable rigidity \cite{Ghys AIF}]  \label{ghys loc thm}
Let $\rho: \Gamma_g \to \Diff^\infty(S^1)$ be smoothly conjugate to a discrete, faithful representation into $\PSL(2,\R)$.   Then any $C^1$ perturbation of  $\rho$ is $C^\infty$ conjugate to a discrete, faithful representation into $\PSL(2,\R)$ .    
\end{theorem}

\noindent This is not quite local differentiable rigidity (in the sense of Definition \ref{rigdef2}), as $\rho'$ is not necessarily $C^\infty$ conjugate to $\rho$.  However, in light of Theorem \ref{mostow circle} Ghys' result is the best one can hope for.  

Though quite technically involved, the idea of Ghys' proof is to improve the topological conjugacy given by Corollary \ref{anosov cor} to a smooth conjugacy by finding a $\rho'$-invariant  projective structure on $S^1$.
Ghys later improved Theorem \ref{ghys loc thm} to a global differentiable rigidity result.    

\begin{theorem}[Global differentiable rigidity \cite{Ghys IHES}]   \label{ghys glob rig}
Let $r \geq 3$, and let $\rho: \Gamma_g \to \Diff^r(S^1)$ be a representation with Euler number $\pm \chi(\Sigma_g)$.  Then $\rho$ is faithful, and conjugate by a $C^r$ diffeomorphism to a representation into a discrete subgroup of $\PSL(2,\R)$.  
\end{theorem}
The essential ``global" ingredient is a theorem of Matsumoto on \emph{maximal representations}, which we'll introduce in the next subsection.   Matsumoto's theorem provides a topological conjugacy between $\rho$ and a representation into $\PSL(2,\R)$; Ghys again finds an invariant projective structure that gives the $C^r$ conjugacy.

Corollary \ref{anosov cor} and Ghys' local differentiable rigidity theorem also apply to the representations to $\PSLk$ constructed in Example \ref{pslk ex}, as the $k$-fold fiberwise cover of the unit tangent bundle of $\Sigma_g$ also admits an Anosov flow with weak-unstable foliation transverse to the fibers.   
In \cite{Bowden}, Bowden uses this to prove a global differentiable rigidity result.  

\begin{theorem}[\cite{Bowden}] \label{bowden thm}
Let $\rho_t: \Gamma_g \to \Diff^\infty(S^1)$ be a continuous path of representations, with $\rho_0$ a discrete, faithful representation to $\PSLk$.    Then $\rho_t$ is smoothly conjugate to $\rho_0$.  
\end{theorem}

Bowden's work produces a topological conjugacy between $\rho_0$ and $\rho_t$, which by Ghys' arguments in \cite{Ghys IHES} must be a smooth conjugacy.   Using our Theorem \ref{invent thm} instead of Bowden's results, one can make the same statement as Theorem \ref{bowden thm} not just for deformations along paths, but for arbitrary representations in the same connected component of $\rho_0$.

\boldhead{Sullivan's ``Hyperbolicity implies structural stability".}

In \cite{Sullivan}, Sullivan gives a completely different definition of ``hyperbolic group action" -- one that applies to actions by \emph{homeomorphisms} on any compact metric space $X$.  
Recently, Kapovich, Leeb and Porti \cite{KLP} have used this property in studying the ``boundary actions" of discrete groups on symmetric spaces -- a situation where the best regularity at hand is really $C^0$, so traditional notions of hyperbolicity have nothing to say.  

We explain Sullivan's idea here, and use it to give an alternative proof of Corollary \ref{anosov cor}, without reference to flows, bundles,  or differentiability.

\begin{definition}[Expanding property \cite{Sullivan}]
A homomorphism $f \in \Homeo(X)$ is \emph{expanding at $x \in X$} if there exists $c > 1$ and a neighborhood $U$ of $x$ (called an expanding region) such that
$$d(f(y), f(z)) > c \, d(y,z)\, \text{ for all } y, z \text{ in } U.$$  
An action $\rho: \Gamma \to \Homeo(X)$ is \emph{expanding} if, for each $x \in X$, there exists some $\rho(\gamma)$ that is expanding at $x$.  
\end{definition}

Expanding actions allow one to ``code" points of $X$ by sequences in $\Gamma$.   

\begin{construction}[Coding by expanding regions]
Suppose $X$ is compact, and $\rho: \Gamma \to \Homeo(X)$ is expanding.  Let $S$ be a finite generating set for $\Gamma$, large enough so that the expanding regions for elements of $S$ cover $X$.  A \emph{coding} of $x_0 \in X$ is a sequence $s_i \in S$, chosen so that $x_1: = s_1(x_0)$ lies in the expanding region for $s_2$, and inductively, $x_i := \rho(s_i)(x_{i-1})$ lies in the expanding region for $\rho(s_{i+1})$.
\end{construction}

Using the expanding property, the point $x_0$ can be recovered from a coding sequence $\{s_i\}$.  If $\epsilon$ is chosen sufficiently small (depending on the cover of $X$ by expanding sets, and on the expanding constants for the $s_i$), and $U(x_n)$ is any neighborhood of $x_n$ contained in the $\epsilon$-ball about $x_n$, then
$$x_0 = \bigcap_n  \rho(s_n ...s_1 s_0)^{-1} (U(x_n)).$$

Sullivan's definition of a \emph{hyperbolic} action is one where this coding is essentially unique.  

\begin{definition}[Hyperbolicity \`a la Sullivan] 
An action is called \emph{hyperbolic} if it is expanding, and the following condition holds:
\begin{quote}
\textit{$(\ast)$ There exists $N$ such that, if $\{s_i\}$ and $\{t_i\}$ are coding sequences for a point $x$, each initial segment $s_1 s_2... s_m$ has a counterpart segment $t_1 t_2... t_k$ such that 
$(s_1... s_m)(t_1 ... t_k)^{-1}$
can be represented by a word of length $N$ in $S$.}
\end{quote}
\end{definition}

This brings us to Sullivan's structural stability theorem.  

\begin{theorem}  \label{sullivan stab}
Let $\Gamma$ be a finitely generated group, and $X$ a manifold.  Any hyperbolic $\rho \in \Hom(\Gamma, \Diff^1(X))$ is structurally stable.    
\end{theorem}

More generally, if $X$ is any compact metric space, $\rho: \Gamma \to \Homeo(X)$ is hyperbolic, and $\rho'$ is a small perturbation of $\rho$ such that a set of expanding generators for $\Gamma$ remain expanding on the same covering sets, then Sullivan's proof of Theorem \ref{sullivan stab} shows that $\rho'$ is conjugate to $\rho$.   This happens, for example, if the generators of $\rho$ are perturbed by small bi-lipschitz homeomorphisms.

The proof of Theorem \ref{sullivan stab} is remarkably simple.  For each point $x \in X$, one chooses a coding sequence 
$\{s_i\}$ and small balls $B_\epsilon(x_i)$ about $x_i$.  Define $h \in \Homeo(X)$ by 
$$h(x) = \bigcap_n  \rho'(s_0 s_1... s_n)^{-1} (B_\epsilon(x_n)).$$
The requirement that $\rho$ be expanding implies that $h(x)$ is a singleton, and property $(\ast)$ implies that it is independent of the sequence $\{s_i\}$.   One then checks in a straightforward manner that $h$ is continuous, injective, and defines a conjugacy.

As any discrete, faithful representation $\rho: \Gamma_g \to \PSL(2,\R) \subset \Homeo(S^1)$ is hyperbolic, we have the following corollary. 

\begin{corollary}[Structural stability, \cite{Sullivan}.]
Let $\rho: \pi_1(M) \to \Homeo(S^1)$ be faithful, with image a discrete subgroup of $\PSL(2,\R)$.   Then any sufficiently small $C^1$ or small bi-lipschitz perturbation of $\rho$ is conjugate to $\rho$. 
\end{corollary}

\subsection{Maximality and geometricity}   \label{max subsec}

We turn now to our second theme \textit{maximality and geometricity}, starting with a particular consequence of Theorem \ref{goldman thm}.

\begin{theorem}[Goldman \cite{Goldman}, \cite{Goldman thesis}] \label{goldman rig thm}
If $\gamma \in \Hom(\Gamma_g, \PSL(2,\R))$ has (maximal) Euler number $\pm \chi(\Sigma_g)$, then it is discrete and faithful and defines a hyperbolic structure on $\Sigma_g$.  
\end{theorem}

This theme -- that representations maximizing some characteristic number should have nice geometric properties -- has been extended to representations of surface groups into many other Lie groups.    A semi-simple Lie group $G$ is said to be of \emph{Hermitian type} if it has finite center, no compact factors, and the associated symmetric space $X$ has a $G$-invariant complex structure.  
This structure allows one to assign a characteristic \emph{Toledo number} $T(\rho)$ to each representation $\rho: \Gamma_g \to G$. Analogous to the Euler class and Euler number for circle bundles, it satisfies an inequality
$| T(\rho) | \leq \chi(\Sigma_g) \mathrm{rank}(X).$ (Proofs of special cases of this inequality were given by Turaev, and Domic and Toledo; the general case is due to Clerc and \O rsted.)

Toledo showed that, in the special case $G = \SU(n,1)$, representations maximizing this invariant have a nice geometric property -- they \emph{stabilize a complex geodesic}, i.e. are conjugate into $\mathrm{S}(\mathrm{U}(n-1) \times \mathrm{U}(1,1))$.  This result was gradually improved by a number of authors, culminating in the following theorem of Burger, Iozzi and Wienhard.  

\begin{theorem}[``Maximal implies geometric" \cite{BIW}] \label{BIW thm}
Let $G$ be a Hermitian Lie group.  If $\rho: \Gamma_g \to G$ has maximal Toledo invariant, then $\rho$ is injective and has discrete image.   Moreover, there is a concrete structure theory that describes the image of $\rho$.  
\end{theorem}

We have been intentionally vague about the ``structure theory", as it is somewhat complicated to state.   A major ingredient in the proof is analyzing the Zariski closure of the image of $\Gamma_g$ -- one views $\rho$ as a composition of two ``maximal" representations
$$ \Gamma_g \to H \hookrightarrow G$$
where the image of $\Gamma_g$ is Zariski dense in $H$, and then one studies separately the maximal representations $\Gamma_g \to H$ and the restrictions on the geometry of $H$ as a subgroup of $G$.   
See \cite{BIW}, or for a gentler introduction, the survey papers \cite{BILW} or \cite{BIW survey}.

\boldhead{Maximality in $\Homeo(S^1)$.}
Though Theorem \ref{BIW thm} and its proof appear to rely heavily on the structures of Lie groups , there is a direct analog for representations of surface groups to $\Homeo(S^1)$, due to Matsumoto.  

\begin{theorem}[Matsumoto's ``maximal implies geometric" \cite{Matsumoto}]  \label{mats thm}
Let $\rho: \Gamma_g \to \Homeo(S^1)$ be a representation with (maximal) Euler number $\pm \chi(\Sigma_g)$.  Then $\rho$ is faithful, and semi-conjugate to a representation with image a discrete subgroup of $\PSL(2,\R)$.  
\end{theorem}

Even more surprising, there are a number of parallels between Matsumoto's proof and the proof of Theorem \ref{BIW thm} in \cite{BIW}.  
In \cite{Iozzi}, Iozzi applies tools developed to study representations of Lie groups (e.g. boundary maps and continuous bounded cohomology) to give a new proof of Matsumoto's theorem, and a unified approach to studying maximal representations into $\SU(n,1)$ and $\Homeo(S^1)$.

\boldhead{Rigidity from geometricity}
As all discrete, faithful representations into $\PSL(2, \R)$ are conjugate by homeomorphisms of the circle (c.f. Section \ref{mostow subsec}), this gives the following ``full rigidity" theorem.  

\begin{corollary}  \label{mat rig cor}
Representations in $\Hom(\Gamma_g, \Homeo(S^1))$ with maximal Euler number are fully rigid.  
Equivalently, the sets of representations with Euler number $\chi(\Sigma_g)$ and with Euler number $-\chi(\Sigma_g)$ each consist of a single semi-conjugacy class.  
\end{corollary}

It is our contention that the phenomenon underlying this rigidity is geometricity --  Matsumoto's theorem states that maximal representations are \emph{geometric} in the sense of Definition \ref{geom def}.  The following theorem was proved in the case of surface groups in \cite{Invent}.    

\begin{theorem}[``Geometricity implies rigidity" \cite{Invent}]  \label{invent thm}
Let $\rho: \Gamma \to \Homeo(S^1)$ be a geometric representation.  Then $\rho$ is fully rigid.  
\end{theorem}

\noindent This theorem not only implies Corollary \ref{mat rig cor}, but also gives rigidity for the representations of surface groups constructed in Example \ref{pslk ex}, and rigidity of representations of other lattices in $\PSLk$.    We also conjecture that geometricity is \emph{the only} source of full rigidity for actions of surface groups.  

\begin{conjecture} \label{invent conj}
Suppose that $\rho: \Gamma_g \to \Homeo(S^1)$ is fully rigid.  Then $\rho$ is geometric.  
\end{conjecture}

In the next two sections (\ref{rot sec} and \ref{CW sec}) we introduce the tools needed to prove Theorem \ref{invent thm} -- which are useful and beautiful in their own right -- and discuss a number of other applications of these tools.     An outline of the proof of Theorem \ref{invent thm} is given in Section \ref{pf sec}, and a hint at the conjecture in Section \ref{flex sec}.

\subsection{Not just surface groups} \label{not just subsec}

From surfaces, one can build many other interesting actions of groups on the circle.   A particularly nice example is the most basic case of the ``universal circle" action, due to Thurston.    Suppose that $\Gamma$ is the fundamental group of a hyperbolic 3-manifold that fibers over the circle, with fiber $\Sigma_g$.   Then $\Gamma$ has a presentation 
$$\Gamma = \langle \Gamma_g, t \mid t \gamma t^{-1} = \phi(\gamma) \rangle$$
for some $\phi \in \Out(\Gamma_g)$.  

If $\rho: \Gamma_g \to \PSL(2,\R)$ is a representation with Euler number $-2g+2$, then it can be extended in a natural way to a representation of $\Gamma$.   To see this, we use the hyperbolic structure on $\Sigma_g$ given by $\rho$.  Let $f$ be a diffeomorphism of $\Sigma_g$ inducing $\phi$ on $\Gamma_g$ (such a diffeomorphism exists since $\Sigma_g$ is a $K(\pi,1)$ space), and lift $f$ to a diffeomorphism $\hat{f}$ of the universal cover $\H^2$ of $\Sigma_g$ that is equivariant with respect to the action of $\Gamma_g$ on $\H^2$.   The diffeomorphism $\hat{f}$ admits a unique extension to a homeomorphism of the circle at infinity $\del \H^2 = S^1$, and we define $\rho(t)$ to be this homeomorphism.   That $f$ is compatible with the action of $\Gamma_g$ on $\H^2$ implies that $\rho$ satisfies the group relations $\rho(t \gamma t^{-1}) = \rho(\phi(\gamma))$, so $\rho$ is well defined.  
Moreover, this extension of $\rho|_{\Gamma_g}$ to a representation $\rho$ of $\Gamma$ is unique.  One way to see this is by considering fixed points of elements $\rho(\gamma)$ for $\gamma \in \Gamma_g$.  The relation $t \gamma t^{-1} = \phi(\gamma)$ implies that $\rho(t)$ necessarily sends the (unique) attracting fixed point of $\rho(\gamma)$ to the attracting fixed point of $\phi(\gamma)$.  Attracting fixed points of elements of $\rho(\Gamma) \subset \PSL(2,\R)$ are dense in $S^1$, so this property determines $\rho(t)$.    Since $\rho$ was \emph{maximal} on $\Gamma_g$, it is fully rigid, and therefore this unique extension is also rigid.   

For readers familiar with the Thurston norm on homology of a 3-manifold (c.f. \cite{Thurston norm}), the construction above can be summarized as follows.  
\begin{theorem}  \label{fibered face thm}
Let $M$ be a hyperbolic 3-manifold that fibers over the circle, and let $\Gamma = \pi_1(M)$.  For each fibered face $F$ of the Thurston norm ball of $H_2(M; \R)$, there is a representation $\rho_F: \Gamma \to \Homeo_+(S^1)$ that is \emph{fully rigid} in $\Hom(\Gamma, \Homeo(S^1))$.
\end{theorem}

Furthermore, Ghys' differentiable rigidity (Theorem \ref{ghys glob rig}) together with Theorem \ref{mostow circle} 
implies that the component of $\Hom(\Gamma, \Homeo_+(S^1))$ containing $\rho_F$ contains no representation with image in $\Diff^3(S^1)$.   

A typical fibered 3-manifold $M$ fibers over the circle in many different ways, giving different, non semi-conjugate, fully rigid actions of $\pi_1(M)$ on $S^1$.   In the language above, if $F$ and $F'$ are two different fibered faces of the Thurston norm ball, then $\rho_{F}$ and $\rho_{F'}$ are not semi-conjugate.  

The construction has been generalized to give actions of fundamental groups of other 3-manifolds.  Instead of a fibration of $M$ over $S^1$, one can use a \emph{pseudo-Anosov flow} on $M$ (\cite{Barbot}, \cite{Fenley}), a \emph{quasi-geodesic flow} \cite{Calegari GT}, or even certain types of \emph{essential lamination} \cite{CD}, to get analogous faithful actions of $\pi_1(M)$ on $S^1$.   It would be interesting to investigate rigidity of these examples using the techniques described below.

\section{Rotation number theory}  \label{rot sec}

One approach to understanding the topology of $\Hom(\Gamma, G)$  is to look for natural ``coordinates" on the space --
this approach has been very fruitful when $G$ is a linear group.  The perspective we promote here, originally due to Calegari (see e.g. \cite{CW}), is that the representation spaces $\Hom(\Gamma, \Homeo(S^1))$ \emph{also} have natural coordinates, given by \emph{rotation numbers}.  
To motivate our discussion, we begin with a few words on the linear case.  

\boldhead{Motivation: trace coordinates on character varieties.}
It is a well-known fact that linear representations of a group $\Gamma$ over a field of characteristic zero are isomorphic if and only if they have the same characters.  In many situations, this fact can be turned into a system of coordinates on the space of representations of $\Gamma$.   Consider $G = \SL(2,\C)$ as a concrete example.  When $\Gamma$ is a finitely generated group, irreducible representations $\Gamma \to \SL(2,\C)$ are determined up to conjugacy by the traces of \emph{finitely many} elements; so a finite collection of traces can be used to parametrize $\Hom(\Gamma, \SL(2,\C)) / \SL(2,\C)$, at least away from the reducible points.  The more sophisticated algebraic-geometric quotient of $\Hom(\Gamma, \SL(2,\C))$ hinted at in Section \ref{char var subsec} that gives the $\SL(2,\C)$ ``character variety" for $\Gamma$ collapses reducible representations of the same character to points, and so traces really \emph{do} give coordinates on this variety.   
This perspective was originally promoted by Culler and Shalen \cite{CS}, who gave an algebraic compactification of components of the $\SL(2,\C)$ character variety when $\Gamma$ is a hyperbolic 3-manifold group.  
A more recent, and perhaps more direct application of trace coordinates can be found in \cite{Goldman trace}.  Here coordinates are used to study the action of the mapping class group of $\Sigma$ on $\Hom(\pi_1(\Sigma), \SL(2,\C)) / \SL(2,\C)$ in the case where $\Sigma$ is a once punctured torus.  This paper also contains many pictures of ``slices" of the character variety (at least the real points!) drawn in trace coordinates.

\boldhead{Characters for $\Homeo(S^1)$.}
Motivated by the linear case, in order to parametrize representations into $\Homeo(S^1)$ up to conjugacy we should look for a class function on $\Homeo(S^1)$ to play the role of the character.   As trace also captures significant dynamical information (e.g. the trace of an element of $\PSL(2,\C)$ determines the translation length of its action on hyperbolic 3-space), we would hope to find a class function on $\Homeo(S^1)$ that holds some dynamical information as well.  

Fortunately, such a function exists: this is the \emph{rotation number} of Poincar\'e.

\subsection{Poincar\'e's rotation number}  \label{rot subsec}
Let $\Homeo_\Z(\R)$ denote the group of orientation-preserving homeomorphisms of $\R$ that commute with integer translations.  $\Homeo_\Z(\R)$ can be thought of as the group of all lifts of orientation-preserving homeomorphisms of $S^1 = \R/\Z$ to its universal cover $\R$, and fits into a central extension
$$0 \to \Z \to \Homeo_\Z(\R) \to \Homeo(S^1) \to 1$$ 
where the inclusion of $\Z$ is as the group of integer translations.  

\boldhead{Notational remark.} 
It is standard to use the notation $\tilde{g}$ for a lift of $g \in \Homeo(S^1)$ to $\Homeo_\Z(\R)$.   However, occasionally -- especially in later sections -- we will want to emphasize the perspective that $\Homeo_\Z(\R)$ should be considered the primary object, and $\Homeo(S^1)$ its quotient.  In this case will use $g$ to refer to an element of $\Homeo_\Z(\R)$ and $\bar{g}$ its image in $\Homeo(S^1)$.

\begin{definition}[Poincar\'e]  
Let $\tilde{g} \in \Homeo_\Z(\R)$ and $x \in \R$.  The \emph{(lifted) rotation number} $\rott(\tilde{g})$ is given by
$$\rott(\tilde{g}) := \lim \limits_{n \to \infty} \frac{\tilde{g}^n(x)}{n}.$$

For $g \in \Homeo(S^1)$, the $\R/\Z$-valued \emph{rotation number} of $g$ is defined by 
$$\rot(g) := \rott(\tilde{g}) \mod \Z$$  
where $\tilde{g}$ is any lift of $g$ to $\Homeo_\Z(\R)$.  
\end{definition} 

It is a classical result that these limits exist and are independent of choice of point $x$.   An elegant and elementary proof is given in \cite{RT} -- the key is the observation that $\tilde{g} \in \Homeo_\Z(\R)$ implies that 
\begin{align}  \label{bounded eq}
\left| \tilde{g}^n(x) -x - (\tilde{g}^{n}(y) - y) \right| \leq 1
\end{align} 
and that this both implies that the sequence $\frac{\tilde{g}^n(x)}{n}$ is Cauchy, hence converges, \emph{and} that its limit is independent of $x$.  
In fact, this proof applies not only to elements of $\Homeo_\Z(\R)$ but to a wider class of maps.  

\begin{definition}[Monotone maps] \label{monotone def}
A \emph{monotone map} of $\R$ is a (not necessarily continuous) map $f: \R \to \R$ satisfying 
\begin{enumerate}[i)]
\item \emph{$\Z$-periodicity}: $f(x + 1) = x+1$ for all $x \in \R$ and 
\item  \emph{monotonicity}: $x \leq y \iff f(x) \leq f(y)$, for all $x, y \in \R$.
\end{enumerate}  
\end{definition} 
\noindent Perhaps these would be better called \emph{$\Z$-periodic monotone maps} so as to emphasize both conditions.  We hope the reader will forgive us for choosing the shorter name, and keep condition i) in mind.

As advertised,  rotation number is a class function that captures a wealth of dynamical information. 
The following properties (perhaps with the exception of continuity) are relatively easy to check just using the definition. 

\begin{proposition}[First properties of rot]  \label{rot prop}
Let $g \in \Homeo_\Z(\R)$. 
\begin{enumerate} 
\item Let $T_{\lambda}$ denote translation by $\lambda$.  Then $\rott(T_\lambda) = \lambda$.  
\item (Homogeneity) $\rott(g^n) = n\rott(g)$.
\item (Conjugacy invariance) $\rott(hgh^{-1}) = \rott(g)$ for any $h \in \Homeo_\Z(\R)$
\item (Quasi-additivity) $| \rott(fg)- \rott(f) -\rott(g) | \leq 2$
\item (Continuity) $\rott$ is a continuous function. 
\item (Periodic points) $\rott(g) = 0$ if and only if $g$ has a fixed point.   More generally, $\rott(g) = p/q \in \Q$ if and only if there exists $x$ such that $g^q(x) = x+p$, in this case the projection of $g$ to $\Homeo(S^1)$ has a periodic orbit of period $q$.  
\end{enumerate}
These statements also hold for $\rot(g)$ when $g \in \Homeo(S^1)$ and 1-3 are easily seen to hold when $g$ is a monotone map.  
\end{proposition}

As an example of using the definition, we give a quick proof of quasi-additivity\footnote{With a more involved argument, it is possible to reduce the bound from 2 to 1 -- we'll see this as a consequence of Theorem \ref{CW thm} below}.

\begin{proof}[Proof of quasi-additivity]
By composing with integer translations, which commute with $f$ and $g$ and do not affect the inequality, we may assume without loss of generality that $f(0) \in [0,1)$ and $g(0) \in [0, 1)$.  This gives $0 \leq fg(0) <  2$.  Also, using $x = 0$ in the definition of $\rott$ gives 
the estimates $\rott(f) \in [0,1]$, $\rott(g) \in [0, 1]$, and $\rott(fg)\in [0, 2]$, from which the desired inequality follows. 
\end{proof}

Although less obvious than the statement about periodic points in Proposition \ref{rot prop}, we can still extract dynamical information from the rotation number when $\rott(g) \notin \Q$.  For this, we need to (as promised, finally!) properly introduce the notion of \emph{semi-conjugacy}.  Semi-conjugacy will play a central role in the rest of this paper -- in fact, we hope to convince the reader that it is the ``right" notion of equivalence for actions by homeomorphisms on the circle.

\subsection{Semi-conjugacy}   \label{SC subsec}
In section \ref{intro rig sec}, we loosely defined a semi-conjugacy to be a kind of ``conjugacy with Denjoy-ing allowed."  We now make this idea precise, eventually giving two equivalent definitions. The first, due to Ghys, comes from the beautiful insight that things are much simpler when lifted to $\R$.

\begin{definition}[Semi-conjugacy \`a la Ghys]  \label{ghys sc def}
Let $\Gamma$ be any group. Two representations $\rho_1$, $\rho_2: \Gamma \to \Homeo_\Z(\R)$ are \emph{semi-conjugate} if there exists a monotone map $h: \R \to \R$ such that 
$$h \rho_1(\gamma) (x) = \rho_2(\gamma) h(x) \,\, \text{ for all } x \in \R \text{ and } \gamma \in \Gamma$$
\end{definition}

\noindent The  map $h$ is called a \emph{semi-conjugacy} between $\rho_1$ and $\rho_2$.

Building on this, we say that representations $\rho_1$ and $\rho_2: \Gamma \to \Homeo(S^1)$ are semi-conjugate if there exists a central extension $\hat{\Gamma}$ of $\Gamma$, and semi-conjugate representations $\hat{\rho}_1$ and $\hat{\rho}_2: \hat{\Gamma} \to \Homeo_\Z(\R)$ such that the following diagrams commute for $i = 1, 2$
\begin{displaymath}
    \xymatrix @M=4pt {
  0 \ar[r] & \Z \ar[r] \ar@{=}[d]& \hat{\Gamma} \ar[d]_{\textstyle \hat{\rho}_i} \ar[r] & \Gamma \ar[d]_{\textstyle \rho_i} \ar[r] &1 \\
   0 \ar[r] & \Z \ar[r] & \Homeo_\Z(\R) \ar[r] & \Homeo(S^1) \ar[r] &1
        }
\end{displaymath}
Note that $\hat{\Gamma}$ is uniquely defined, it is the pullback of the central extension 
$\Z \to \Homeo_\Z(\R) \to \Homeo(S^1)$
by $\rho_1$ (and also the pullback by $\rho_2$ -- the definition of semi-conjugacy implies in particular that these pullbacks are isomorphic as extensions of $\Gamma$).

\begin{remark}[A word of warning]
Definition \ref{ghys sc def} appears in \cite{Ghys Lef}, but with inconsistent use of pullbacks to $\Homeo_\Z(\R)$, an unfortunate little mistake which has caused quite a bit of confusion!  To see what goes wrong when one fails to pull back, take $h: \R \to \R$ to be the floor function $h(x) = \lfloor x \rfloor$.  This is a monotone map, and it descends to the constant map $\bar{h}(x) = 0$ on $S^1 = \R/\Z$.  Now any representation $\rho: \Gamma \to \Homeo(S^1)$ appears to be ``semi-conjugate"\footnote{of course, not with the correct definition of semi-conjugacy!} to the trivial representation via $\bar{h}$, as 
$$0 = \bar{h} \rho(\gamma) (x) = \bar{h}(x) \text{ for all } x \in S^1$$
However, if we instead take a (pullback) representation $\hat{\rho}$ to $\Homeo_\Z(\R)$, the situation is different.  Suppose that $\hat{\rho}$ is semi-conjugate to a pullback $\hat{\rho}_{\mathrm{triv}}$ of the trivial representation via $h$.  Up to multiplying by an element of the center of $\hat{\Gamma}$, each $\gamma \in \hat{\Gamma}$ satisfies 
$\hat{\rho}_{\mathrm{triv}}(\gamma) = \id$. In this case, we have 
$$\lfloor \hat{\rho}(\gamma)(x) \rfloor = \lfloor x \rfloor \text{ for all } x \in \R.$$ 
In particular, taking $x \in \Z$ gives that $\hat{\rho}(\gamma)^{\pm 1}(x) \geq x$, which implies that $x$ is fixed by $\hat{\rho}(\gamma)$.  Thus, $\rho(\Gamma)$ fixes $0 \in S^1$.  With some work, this argument can be improved to show generally that a representation $\rho: \Gamma \to \Homeo(S^1)$ is semi-conjugate to the trivial representation if and only if it has a global fixed point.  
\end{remark}

One of the benefits of lifting to the line is that it makes it easy to show semi-conjugacy is an equivalence relation.  

\begin{proposition} \label{equiv prop}
Semi-conjugacy is an equivalence relation
\end{proposition}

\begin{proof}(Following \cite[Prop 2.1]{Ghys Lef}).
It suffices to prove that semi-conjugacy is an equivalence relation on representations to $\Homeo_\Z(\R)$, since in the case of representations to $S^1$, the extension $\hat{\Gamma}$ is uniquely defined.  

For representations to $\Homeo_\Z(\R)$, reflexivity of semi-conjugacy is immediate, and transitivity an easy exercise.  For symmetry, suppose that $\rho$ and $\rho': \Gamma \to \Homeo_\Z(\R)$ satisfy $\rho(\gamma) h = h \rho'(\gamma)$ for some $h: \R \to \R$ as in Definition \ref{ghys sc def}.  Define $h': \R \to \R$ by 
$$h'(x) = \sup \{y \in \R \mid h(y) \leq x \}.$$
Then $h'(x+1) = \sup \{y \in \R \mid h(y-1) \leq x \} = h'(x) + 1$; and $x_1 \leq x_2$ implies that $h'(x_1) \leq h'(x_2)$, so $h$ is monotone.  
Moreover, by construction we have 
\begin{align*}
h' \rho(\gamma)(x)  &= \sup \{ y \in \R \mid h(y) \leq \rho(\gamma)(x) \} \\
&= \sup \{y= \rho'(\gamma)(z) \in \R \mid h(z) \leq x \}, \text{ since } h \text{ is order preserving} \\
&= \rho'(\gamma)h'(x)
\end{align*}
so $h'$ is a semi-conjugacy between $\rho'$ and $\rho$.
\end{proof}

\boldhead{Semi-conjugacy and minimal actions.} 
The next definition of semi-conjugacy is due to Calegari and Dunfield, originally appearing in \cite[Def. 6.5, Lem. 6.6]{CD}.  
It avoids lifting to $\R$ by describing semi-conjugate representations as those having a kind of ``common reduction".  

\begin{definition}[Degree one monotone map]
A map $h: S^1 \to S^1$ is called a \emph{degree one monotone map} if it admits a continuous lift $\tilde{h}: \R \to \R$ that is \emph{monotone} as in Definition \ref{monotone def}.  
\end{definition}

\begin{definition}[Semi-conjugacy  \`a la Calegari--Dunfield] \label{C-D def} 
Two representations $\rho_1$ and $\rho_2$ in $\Hom(\Gamma, \Homeo_+(S^1))$ are semi-conjugate if and only if there is a third representation $\bar{\rho}: \Gamma \to \Homeo_+(S^1)$ and degree one monotone maps $h_1$ and $h_2$ of $S^1$ such that 
$$\rho_i \circ h_i = h_i \circ \bar{\rho}.$$  
\end{definition}

Moreover, if $\rho_1$ has a finite orbit, then $\bar{\rho}$ can be taken to be conjugate to an action by rotations and $h_1$ a (discontinuous) map with image the finite orbit of $\rho_1$.  In this case, $\rho_2$ necessarily has a finite orbit as well.  
Otherwise, $\bar{\rho}$ can be taken to be a \emph{minimal} action, i.e. with all orbits dense; and $h_i$ a map collapsing each wandering interval for the action of $\rho_i$ to a point.  (See \cite[Prop 5.6]{Ghys Ens} for more justification of this description).  
In this way, one thinks of $\bar{\rho}$ as capturing the essential dynamical information of $\rho_i$.  

This description also emphasizes the fact that any semi-conjugacy between two minimal actions is a genuine conjugacy, and it can be used to show that semi-conjugacy classes are \emph{connected}.

\boldhead{Rotation numbers detect semi-conjugacy.} 
Now we may finally complete our list of properties of the rotation number.  The following proposition is essentially a theorem of Poincar\'e, a proof and further discussion can be found in \cite[Sect. 5.1]{Ghys Ens}.  

\medskip
\noindent \textbf{Proposition \ref{rot prop}, continued} (Further properties of rot).
\begin{enumerate}
\item[5.]  For $g \in \Homeo(S^1)$, $\rot(g)$ is irrational if and only if $g$ is semi-conjugate to an irrational rotation.  This semi-conjugacy is actually a conjugacy if $g$ has dense orbits.  
\item[6.]
 (Detects semi-conjugacy) More generally, $f$ and $g$ in $\Homeo_\Z(\R)$ are semi-conjugate if and only if $\rott(f) = \rott(g)$.
\end{enumerate} 

Property 6 says that rotation number is a complete invariant of semi-conjugacy classes in $\Homeo_\Z(\R)$ or $\Homeo(S^1)$. In other words, an action of $\Z$ is determined up to semi-conjugacy by the rotation number of a generator.  In Section \ref{coord subsec} we will see how to generalize this to actions of other groups, using rotation numbers as ``coordinates" on the quotient of $\Hom(\Gamma, \Homeo(S^1))$ by semi-conjugacy.   But first, we specialize again to surface groups, and return to the definition of the Euler number.

\subsection{The Euler number as a rotation number}  \label{euler subsec}

The goal of this section is to give two concrete descriptions of the Euler number of a circle bundle over a surface.  The first, using rotation numbers, is specific to flat bundles and one of the main ingredients in Milnor's (and Wood's) original proofs of the Milnor--Wood inequality.   The second definition is classical and obstruction-theoretic; our description is intended to make it easy to reconcile the two.  

Consider the standard presentation of a surface group $\Gamma_g = \langle a_1, b_1, ... a_g, b_g \mid \prod_i [a_i, b_i] \rangle.$
We will exploit the fact that this group has a single relator, made up of commutators, to assign an integer to each representation $\rho: \Gamma_g \to \Homeo(S^1)$.  
The reason commutators are useful is that they have \emph{cannonical lifts} to $\Homeo_\Z(\R)$ -- for any $f$ and $g \in \Homeo(S^1)$ with lifts $\tilde{g}$ and $\tilde{f} \in \Homeo_\Z(\R)$, the homeomorphism $[\tilde{f}, \tilde{g}]$ is independent of the choice of lifts.  
We'll use the notation $\rott[f,g]: = \rott([\tilde{f}, \tilde{g}])$, and similarly $\rott\left( [f,g][h, k] \right): = \rott ( [\tilde{f}, \tilde{g}] [\tilde{h}, \tilde{k}] )$ etc. for products.  

The following definition is implicit in \cite{Milnor} and used explicitly in \cite{Wood}.  
\begin{definition}  \label{eu def}
The \emph{Euler number} of a representation $\rho: \Gamma_g \to \Homeo(S^1)$ is the integer
$$\euler(\rho): = \rott \left( \prod_{i=1}^g [\rho(a_i), \rho(b_i)] \right).$$
\end{definition}

Since $\rott$ is continuous, it follows immediately that $\euler$ is continuous on $\Hom(\Gamma_g, \Homeo(S^1))$.  
Of course, this definition only makes sense for flat bundles.  Construction \ref{classical eu} below gives a definition applicable to arbitrary circle bundles over surfaces.  (See also Definition \ref{eu def 1}).  

\begin{construction}[Euler number as an obstruction class]  \label{classical eu}

Build $\Sigma_g$ by taking a single 0-cell $x$, attaching a loop to $x$ for each generator, and then gluing the boundary of a 2-cell to the loop based at $x$ corresponding to the word $\prod [a_i, b_i]$.  Cutting this complex along the 1-skeleton gives the familiar picture of $\Sigma_g$ as a $4g$-gon with sides identified.

Given an $S^1$ bundle $E$ over $\Sigma_g$, there is no obstruction to building a section over the 0-cell, or extending this to a section over the 1-skeleton since $S^1$ is connected.   However, there \emph{is} an obstruction to extending the section over the 2-cell.  
The easiest way to describe this obstruction is to pull back the bundle to an $S^1$ bundle over the universal cover $\tilde{\Sigma}_g \to \Sigma_g$.  Since $\tilde{\Sigma}_g$ is contractible, the pull-back bundle trivializes as $\tilde{\Sigma}_g \times S^1$.  A section $\sigma$ over the 1-skeleton of $E$ corresponds to a section of $\tilde{\Sigma}_g \times S^1$ over the 1-skeleton that is \emph{equivariant} with respect to the action 
of $\pi_1(\Sigma_g)$.   The restriction of such a section to the boundary of a 2-cell in $\tilde{\Sigma_g}$ gives a map $S^1 \to \tilde{\Sigma}_g \times S^1$, and projection to the $S^1$ factor a map $S^1 \to S^1$.  The winding number of this map is the only obstruction to extending it over the 2-cell, and is exactly the Euler number of the bundle.  
\end{construction}

Now our description of a \emph{flat} bundle as a quotient $(\tilde{\Sigma}_g \times S^1)/ \pi_1(\Sigma_g)$ should make it possible to reconcile Construction \ref{classical eu} with Definition \ref{eu def}.  We leave this as an exercise for the reader.   
The truly enthusiastic can also try to reconcile Construction \ref{classical eu} with Definition \ref{eu def 1} (interpreted the right way, the construction gives a means of assigning integers to 2-cells, so is a 2-cocycle...)

\subsection{Rotation numbers as coordinates}  \label{coord subsec} 

In \cite{Ghys Lef}, Ghys gave a cohomological condition on representations in $\Hom(\Gamma, \Homeo(S^1))$ equivalent to the representations being semi-conjugate\footnote{Ghys' condition is that the representations determine the same \emph{bounded integer Euler class} in $H^2_b(\Gamma; \Z)$, we'll say a bit more about this in Section \ref{h2b subsec}}.  
Matsumoto later translated this into a statement about rotation numbers.  

Before stating this condition, note that for any two elements $f$, $g \in \Homeo_+(S^1)$ with lifts $\tilde{f}$ and $\tilde{g}$ in $\Homeo_\Z(\R)$, the number 
$$\tau(f,g):= \rott(\tilde{f} \tilde{g}) - \rott(\tilde{f}) -\rott(\tilde{g})$$
 does not depend on the choice of lifts $\tilde{f}$ and $\tilde{g}$.

\begin{theorem}[``Rotation numbers are coordinates" \cite{Ghys Lef}, \cite{Matsumoto num}] \label{rot coord}
Let $\Gamma$ be a group with generating set $S$.  Two representations $\rho$ and $\rho'$ in $\Hom(\Gamma, \Homeo(S^1))$ are semi-conjugate if and only if the following two conditions hold
\begin{enumerate}[i)]
\item $\rot(\rho(s)) = \rot(\rho'(s))$ for each generator $s \in S$ . 
\item $\tau(\rho(\gamma_1), \rho(\gamma_2)) = \tau (\rho'(\gamma_1), \rho'(\gamma_2))$ for each pair of elements $\gamma_1$ and $\gamma_2$ in $\Gamma$.  
\end{enumerate}
\end{theorem}

Matsumoto's proof consists in showing that the conditions i) and ii) given above imply that $\rho$ and $\rho'$ satisfy the cohomological condition given by Ghys\footnote{The key observation is that $\tau$ is a cocycle representative for the bounded \emph{real} Euler class.}.
We give a self-contained proof in the spirit of Ghys, but without reference to cohomology.  However, the reader familiar with $H^2_b$ should be able to find it lurking in the background.

\begin{proof}[Proof of Theorem \ref{rot coord}]
Suppose that $\rho$ and $\rho'$ satisfy the conditions of Theorem \ref{rot coord}.    
For each generator $s$ of $\Gamma$, pick a lift $\tilde{s}$ of $\rho(s)$ and $\tilde{s}'$ of $\rho(s')$ to $\Homeo_\Z(\R)$ such that 
$$\rott(\tilde{s}) =  \rott(\tilde{s}'),$$
this is possible by condition i).  
Note that condition ii) implies (inductively) that for any string of generators $s_1s_2 ... s_n$ we have
\begin{equation} \label{rot eq}
\rott(\tilde{s_1}\tilde{s_2}...\tilde{s_n}) =  \rott(\tilde{s_1}'\tilde{s_2}'...\tilde{s_n}').
\end{equation}
Let $\hat{\Gamma}$ denote the pullback $\rho^*(\Homeo_\Z(\R))$.  By definition, $\hat{\Gamma}$ is the subgroup of $\Homeo_\Z(\R) \times \Gamma$ generated by $(\tilde{s}, s)$ and $(T, \id)$ where $T(x) = x+1$; and $\hat{\rho}: \hat{\Gamma} \to \Homeo_\Z(\R)$ is just projection onto the $\Homeo_\Z(\R)$ factor.   Equation \eqref{rot eq} implies that the homomorphism $\phi: \hat{\Gamma} \to \rho'^*(\Homeo_\Z(\R))$, given by 
\begin{align*}
& \phi(\tilde{s}, s) = (\tilde{s}', s'), \text{ for } s \in S \\
& \phi(T, \id) = (T, \id)
\end{align*}
is an isomorphism of central extensions -- if some word $(T^k \tilde{s}_{i_1}...\tilde{s}_{i_n},  s_{i_1}...s_{i_n})$ in the generators of $\hat{\Gamma}$ is trivial, then 
$s_{i_1}...s_{i_n} = \id \in \Gamma$, so $\tilde{s}_{i_1}...\tilde{s}_{i_n}$ is an integer translation, and \eqref{rot eq} implies that $\tilde{s}_{i_1}'...\tilde{s}_{i_n}'$ is the same translation.  Hence $(T^k \tilde{s}_{i_1}'...\tilde{s}_{i_n}',  s_{i_1}...s_{i_n})$ is trivial as well.  

Moreover, the representation $\hat{\rho}': \hat{\Gamma} \to \Homeo_\Z(\R)$ defined by $\hat{\rho}'(\tilde{s}, s) = \tilde{s}'$ gives a commutative diagram
\begin{displaymath}
    \xymatrix @M=4pt {
    0 \ar[r] &\Z \ar[r] \ar@{=}[d]& \hat{\Gamma} \ar[d]_{\textstyle \hat{\rho}'} \ar[r] & \Gamma' \ar[d]_{\textstyle \rho'} \ar[r] & 1 \\
     0 \ar[r] &  \Z \ar[r] & \Homeo_\Z(\R) \ar[r] & \Homeo(S^1) \ar[r] & 1
        }
\end{displaymath}
It remains to construct a semi-conjugacy between $\hat{\rho}$ and $\hat{\rho}'$. 
For $x \in \R$, define 
$$h(x) := \sup_{\gamma \in \hat{\Gamma}} \left\{ \hat{\rho}(\gamma)^{-1} \hat{\rho}'(\gamma)(x) \right\}.$$  
To see that this supremum is always finite, note that \eqref{rot eq} implies that $\rott(\hat{\rho}(\gamma)) = \rot(\hat{\rho}'(\gamma))$, so quasi-additivity applied to $\hat{\rho}(\gamma)^{-1}$ and $\hat{\rho}'(\gamma)$ gives the bound $| \rott(\hat{\rho}(\gamma)^{-1} \hat{\rho}'(\gamma) ) | \leq 2$.  The definition of $\rott$ now implies that 
$x-3 \leq \hat{\rho}(\gamma^{-1}) \hat{\rho}'(\gamma)(x) \leq x+3$
and this holds for any $\gamma \in \hat{\Gamma}$.

As in the proof of Proposition \ref{equiv prop}, it is also easy to check that $h$ is a monotone map, and so we just need to verify that $h$ defines a semi-conjugacy.   Let $\alpha \in \hat{\Gamma}$, then
\begin{align*}
h \hat{\rho}'(\alpha)(x) &= \sup_{\gamma \in \hat{\Gamma}} \left\{ \hat{\rho}(\gamma)^{-1} \hat{\rho}'(\gamma)\hat{\rho}'(\alpha)(x) \right\} \\
			&= \sup_{\gamma \in \hat{\Gamma}} \left\{ \hat{\rho}(\gamma)^{-1} \hat{\rho}'(\gamma \alpha)(x) \right\} \\
			&= \sup_{\beta \in \hat{\Gamma}} \left\{ (\hat{\rho}(\beta \alpha^{-1})^{-1} \hat{\rho}'(\beta)(x) \right\} \\
			&= \hat{\rho}(\alpha) \sup_{\beta \in \hat \Gamma} \left\{ \hat{\rho}(\beta)^{-1} \hat{\rho}'(\beta)(x) \right\} \\
			&= \hat{\rho}(\alpha) h(x)
\end{align*}
which is exactly what we needed to show.  
\end{proof}

\begin{remark} 
Theorem \ref{rot coord} has an alternative formulation and proof that also applies to \emph{semigroups}, due to Golenishcheva-Kutuzova, Gorodetski, Kleptsyn, and Volk \cite{GGKV}.    Since any semi-conjugacy between \emph{minimal} actions is a genuine conjugacy, 
another way to put Theorem \ref{rot coord} is to say that the conjugacy class of a minimal action is determined by rotation numbers of generators and the function $\tau$.   In \cite{GGKV} it is shown that, if both the (positive) semigroup generated by a collection $s_i \in \Homeo_+(S^1)$, \emph{and} the (negative) semigroup generated by the $\{s_i^{-1}\}$ act minimally, then the collection of $s_i$ is determined up to conjugacy by their rotation numbers and the restriction of $\tau$ to the semigroup generated by the $s_i$. (It would be nice not to have to consider the negative semigroup at all, but this hypothesis really is necessary -- see Example 1 in \cite{GGKV}.)
\end{remark}

\subsection{Rotation number as a \emph{quasimorphism}}   \label{qm subsec}

Perhaps it sounds obvious when stated this way, but the reason that $\rott$ can capture so much information about homeomorphisms of $S^1$ is because it is \emph{not} a homomorphism.  (In fact, $\Homeo(S^1)$ is a simple group, so any homomorphism to $\R/\Z$ is necessarily trivial.)  Here is an instructive example of how additivity can fail:

\begin{example}
Let $f(x) = \lfloor x \rfloor$ and $g(x) = \lfloor x + 1/2 \rfloor - 1/2$.   Then $f$ and $g$ are monotone maps and each has a fixed point, so $\rott(f) = \rott(g) = 0$.  However 
$fg(0) = -1$ and hence $\rott(fg) = -1$.  

One can modify this example so that $f$ and $g$ are not just monotone maps, but lie in $\Homeo_\Z(\R)$ -- our computation only used the property that $f$ and $g$ both had fixed points, but $fg(x) = x-1$ for some $x \in \R$.  In fact, one can find such $f$ and $g$ that are lifts of hyperbolic elements in $\PSL(2,\R) \subset \Homeo(S^1)$.  
\end{example}

But there is one particular situation where $\rott$ is additive;  we include it here, as it gives a nice reminder of the definition of $\rott$ and because we'll need to use the result later. 
\begin{lemma} \label{additive lem}
Let $f$ and $g \in \Homeo_\Z(\R)$ satisfy $fg = T^k$ where $T^k$ denotes the translation $T^k(x) := x+k$, for $k \in \Z$.   Then $\rott(f) + \rott(g) = k$.  
\end{lemma}  

\begin{proof} 
That $fg = T^k$ implies that $f = T^k \circ g^{-1}$, hence 
$$
\rott(f) = \rott(T^k g^{-1}) = \lim \limits_{n \to \infty} \frac{T^{kn} \tilde{g}^{-n}(x)}{n} = \lim \limits_{n \to \infty} \frac{\tilde{g}^{-n}(x)+ kn}{n} = \rott(g^{-1}) + k = -\rott(g) + k.
$$
\end{proof}

Despite not being a homomorphism, $\rott$ is a ``quasimorphism".
A function $r$ from a group $G$ to $\R$ is called a \emph{quasimorphism} if there exists $D \geq 0$ such that
$$|r(f g) - r(f) -r(g) | < D \text{ for all } f, g \in G.$$
You can think of this as saying that $r$ is ``a homomorphism up to bounded error."  The optimal bound $D$ is called the \emph{defect} of $r$.  
The existence of a bound $D$ for the function $\rott$ is exactly the quasi-additivity property of Proposition \ref{rot prop}.  In fact, the following is true.  

\begin{proposition}[Prop 5.4 in \cite{Ghys Ens}] \label{unique qm}
Rotation number is the \emph{unique} quasimorphism from $\Homeo_\Z(\R)$ to $\R$ that is a homomorphism when restricted to one generator groups and that takes the value 1 on translation by 1.  
\end{proposition} 

While finding the defect of $\rott$ is not too difficult (spoiler: it's $D=1$), it is a different question to find the optimal bound for $|\rott(fg) - \rott(f) -\rott(g) |$ if $f$ and $g$ have \emph{specified rotation numbers}.  

\begin{problem} \label{rot prob}
Let $f$ and $g \in \Homeo_\Z(\R)$.  Given $\rott(f)$ and $\rott(g)$, what are the possible values of $\rott(fg)$?  More generally, if $w$ is a word in $f$ and $g$, what are the possible values of $\rott(w)$?  
\end{problem}

This was for many years an open problem, with significant implications in topology. This connection to topology was through the classification of taut foliations on Seifert fibered 3-manifolds.  (This should, perhaps, not strike the reader as not so surprising, given the relationship between foliations and group actions discussed in Section \ref{bundle subsec}.)
Specifically, in \cite{EHN}, Eisenbud, Hirsch and Neumann reduced the last step of the program to classify such foliations to solving Problem \ref{rot prob} for the word $fg$.  A conjectural answer to this question, and quite a bit of evidence, was later given in \cite{JN}, and the problem finally solved by Naimi \cite{Naimi}, more than 10 years after Eisenbud, Hirsch and Neumann's work.

More recently, Calegari and Walker \cite{CW} developed an algorithm which answers Problem \ref{rot prob} not only for $fg$ but for any word in $f$ and $g$.  This algorithm is elegant and elementary, implementable by computer or by hand, and the main tool used to prove our rigidity Theorem \ref{invent thm}.

\subsection{Quasimorphisms and $\H^2_b(G; \R)$}  \label{h2b subsec}

Before describing the algorithm of Calegari--Walker and the solution to Problem \ref{rot prob}, we make a quick comment on the connections between quasimorphisms and \emph{bounded cohomology}.   Section 6 of \cite{Ghys Ens} provides a wonderful introduction to the subject and its applications to groups acting on the circle; we see no need to compete with it here so will be very brief in our exposition.  

\boldhead{Bounded cohomology.}
One construction of the cohomology $H^*(G; \R)$ of a discrete group is as the cohomology of the complex $C^n(G; \R)$ of functions $G^n \to \R$, equipped with a suitable differential.  Bounded cohomology is the cohomology of the subcomplex $C^n_b(G; \R)$ of \emph{bounded} functions.   The resulting theory is remarkably different.  For example, amenable groups always have trivial bounded cohomology in all degrees, while hyperbolic groups (including our favorite $\Gamma_g$) have infinite dimensional second bounded cohomology.  

A more sophisticated way to give the definition of a quasimorphism on a group $G$ is as a ``inhomogeneous real 1-cochain on $G$, whose coboundary is bounded."   The coboundary of a quasimorphism $r: G \to \R$ is thus a \emph{bounded 2-cocycle} on $G$, representing an element of second bounded cohomology, $H^2_b(G; \R)$.   In the case of the quasimorphism $\rott$, the boundary 2-cocycle is exactly the function $\tau$ from Theorem \ref{rot coord}.    Many other results on rotation numbers of group actions can also be rephrased in terms of bounded cohomology.  For instance, Proposition \ref{unique qm} can be interpreted as a statement that $H^2_b(\Homeo(S^1); \R)$ is one dimensional, 
and Ghys' original statement of Theorem \ref{rot coord} is a statement about the \emph{bounded integer Euler class} in $H^2_b(\Gamma_g; \Z)$.   
See  \cite{Ghys Ens} and \cite{Ghys Lef} for more details.

\section{Rotation numbers of products}  \label{CW sec}

In this section we explain the algorithm of Calegari and Walker and its first applications, including a solution to Problem \ref{rot prob} and a short proof of the Milnor--Wood inequality.

As a starting point, we return to the basic question raised in Section \ref{qm subsec}.
\begin{quote}
\textit{Given $\rott(f)$ and $\rott(g)$, and a word $w$ in $f$ and $g$, what are the possible values of $\rott(w)$?}
\end{quote}
We will consider an even more general version of this problem 
\begin{quote}
\textit{Given constraints $\rott(f_1) = s_1, ... , \rott(f_n) = s_n$, and a word $w$ in $f_1, f_2, ... , f_n$, what are the possible values of $\rott(w)$?}
\end{quote}

Our first lemma reduces this to a much easier problem.  The lemma -- like much of the material in this section -- originally appears in \cite{CW}.  

\begin{lemma}  \label{R bounded lem}
Let $w$ be any word in $f_1, ... , f_n$.  The set $\{ \rott(w) \mid \rott(f_i) =s_i \}$ is bounded and connected, hence an interval.   
\end{lemma} 

\begin{proof}  In the $n = 2$ case, boundedness follows immediately from the fact that $\rott$ is a quasimorphism.  The general case is proved by an inductive argument.   
For connectedness, property 6 of Proposition \ref{rot prop} states that any set of the form $\{f \in \Homeo_\Z(\R) \mid \rott(f) = s\}$ is a single semi-conjugacy class, hence connected.  Since $\rott$ is continuous, $\{ \rott(w) \mid \rott(f_i) =s_i \}$ is connected.  
\end{proof}

Thus, to determine all possible values of $\rott(w(f_1, ... , f_n))$ given the constraints $\rot(f_i) = s_i$, it suffices to determine the maximum and minimum possible values.  
To this end, let 
$$R_w(s_1, ... , s_n) := \sup \{\rott(w) \mid \rott(f_i) = s_i\}.$$  
We leave it to the reader to check that $\inf \{\rott(w) \mid \rott(f_i) = s_i \} = -R_w(-s_1, ... , -s_n)$, so a study of $R_w(s_1, ...  ,s_n)$ suffices.   In \cite{CW}, a compactness argument is used to show that the supremum is actually a maximum, although we won't need this fact here.  Using continuity of $\rott$, they also show the following.  

\begin{proposition}[see Lemmas 2.14 and 3.3 in \cite{CW}]  \label{rat s prop}
$R_w$ satisfies the following properties
\begin{enumerate}
\item (Lower semi-continuity) Let $s_i(k)$ be a sequence approaching $s_i$. Then 
$$R_w(s_1, ... , s_n) \geq \limsup R_w(s_1(k),..., s_n(k)).$$   

\item (Monotonicity) If $s_i \geq s_i'$, then 
$$R_w(s_1, ..., s_n) \geq R_w(s_1',..., s_n').$$
\end{enumerate}
\end{proposition} 

In fact, lower semi-continuity and the argument that $\{ \rott(w) \mid \rott(f_i) =s_i \}$ is an interval apply not only to words in the letters $f_1, ... , f_n$, but to words that include their inverses $f_1^{-1}, ... f_n^{-1}$ also.   However, monotonicity fails in the case where inverses are permitted.  

Going forward, we will make frequent use of monotonicity, and so never permit inverses.  Calegari--Walker emphasize this by referring to a word in $n$ letters as a \emph{positive word}, we will keep our convention and just say ``word" or ``word in $n$ letters".  

An nice consequence of Proposition \ref{rat s prop} is the following.
\begin{corollary}   \label{rat s cor}
Let $w$ be a word in $f_1, ... ,  f_n$.  Then $R_w(s_1, ..., s_n)$ is completely determined by its values for rational $s_i$.  
\end{corollary}

\begin{proof} This is just the standard statement that an increasing, lower semi-continuous function is determined by its values on a dense set.  
\end{proof}

\subsection{The Calegari--Walker algorithm}
The Calegari--Walker algorithm is a process to compute $R_w(s_1, ... , s_n)$ for rational $s_i$.  
It takes as input a possible cyclic order of a set of fixed or periodic points for homeomorphisms $f_i$, and as output gives the maximum possible value of $\rott(w(f_1, ... , f_n))$ where $f_i$ are homeomorphisms subject to these constraints.   By running the algorithm on all of the (finitely many) possible cyclic orderings of periodic orbits with a given period $s_i$, and taking the maximum of the outputs, one recovers $R_w$.  

The description of the algorithm in \cite{CW} is designed to be easily implementable by computer.  We have a different aim here, so will describe it in slightly different language.  The main idea is remarkably simple -- one simply replaces homeomorphisms with \emph{monotone maps} to reduce the computation to a finite problem.  To explain this precisely requires a small amount of set-up, which we do now.

\begin{definition}[``Periodic" elements of $\Homeo_\Z(\R)$] \label{per def}
An element $f \in \Homeo_\Z(\R)$ is \emph{$p/q$--periodic} if $\rott(f) = p/q \in \Q$.  
\end{definition}
\noindent In other words, ``periodic" elements of $\Homeo_\Z(\R)$ are lifts of homeomorphisms of $S^1$ with periodic points.  

\begin{definition}[Periodic orbits in $\R$]  \label{per orbit def}
Let $f \in \Homeo_\Z(\R)$ be $p/q$--periodic.  
If $p/q$ is in lowest terms\footnote{Our convention is that ``in lowest terms" implies that $q>0$. We also say $0/1$ is the expression of 0 in lowest terms.}, we define a $p/q$-\emph{periodic orbit} for $f$ to be any orbit $X$ that projects to a periodic orbit for $\bar{f} \in \Homeo(S^1)$.  

If instead $p/q = k p'/k q'$ where $p'/q'$ is in lowest terms, then a $p/q$-periodic orbit is defined to be a union of $k$ disjoint $p'/q'$-periodic orbits for $f$.
\end{definition}

\noindent Much like the definition of rotation number, Definitions \ref{per def} and \ref{per orbit def} naturally extend to all \emph{monotone maps} of $\R$.  

The next definition describes periodic monotone maps that translate points ``as far as possible to the right", given a constraint on an orbit.   

\begin{definition}[Maximal monotone map]  \label{max mon def}
Let $p/q \in \Q$, with $q>0$. Let $X \subset \R$ be a lift of a set $\bar{X} \subset S^1$ with $|\bar{X}| = q$, enumerated in increasing order as
$... < x_{i-1}  <  x_i < x_{i+1} < ...$ \\
The \emph{(X, p/q) maximal monotone map} is the map $f: \R \to \R$ defined by 
$$f(x) = x_{i+p} \text{ for } x \in (x_{i-1}, x_i].$$
\end{definition}

We'll call a set $X \subset \R$ as described in Definition \ref{max mon def} a \emph{$\Z$-periodic set of cardinality $q$}. 
\noindent Note that the enumeration of $X$ satisfies $x_{i+q} = x_i + 1$, and that the $(X, p/q)$ maximal monotone map satisfies 
\begin{enumerate}[i.]
\item ($\Z$-periodicity) $f(x+1) = f(x) + 1$, so it really is a \emph{monotone map}, and 
\item ($p/q$-periodic orbits) $f^q(x_i) = x_{i+pq} = x_i + p$.
\end{enumerate}

The following lemma explains the terminology ``maximal,"  its proof is immediate from the definition.  
\begin{lemma} \label{max mon lem}
Let $X$ be a $\Z$--periodic set of cardinality $q$, and let $f$ be the $(X, p/q)$  maximal monotone map.  If $g$ is any $p/q$--periodic monotone map with $X$ as a $p/q$-periodic orbit, then 
$$g(x) \leq f(x) \text{ for all } x \in \R.$$
\end{lemma} 

As a consequence, we have the following.  
\begin{proposition}  \label{max prop}
Let $p_i/q_i$ be given, and let $f_i$ be $(X_i, p_i/q_i)$ maximal monotone maps.  If $g_i$ are any other $(X_i, p_i/q_i)$ monotone maps, and  $w$ a word in $n$ letters, then 
$$\rott(w(g_1, ... g_n)) \leq \rott(w(f_1, ... f_n)).$$
\end{proposition}

\begin{proof} 
Lemma \ref{max mon lem} implies that $g_i(x) \leq f_i(x)$ for all $x \in \R$, hence $w(g_1, ... g_n)(x) \leq w(f_1, ... f_n)(x)$, and so 
$\rott(w(g_1, ... g_n)) \leq \rott(w(f_1, ... f_n))$.
\end{proof}

\begin{proposition}  \label{Q prop}
Let $f_i$ be $(X_i, p_i/q_i)$ maximal monotone maps, and let $w$ be a word in the $f_i$.  Then 
$\rott(w) \in \Q$.
\end{proposition}

\begin{proof}
Let $\mathbf{X} := \bigcup_i X_i$, and note that $\mathbf{X}$ is invariant under integer translations, and that $w(\mathbf{X}) \subset \mathbf{X}$.  
Since $\mathbf{X}$ mod $\Z$ is finite, it follows that it contains some \emph{periodic point} for $w$, i.e. a point $x$ such that 
$$w^m(x) = x+n$$
for some $m, n \in \Z$.  From the definition of rotation number, we have $\rott(w) = m/n$.   
\end{proof}

Proposition \ref{Q prop} also gives an easy and effective way to compute $\rott(w)$ given $(X_i, p_i/q_i)$.   Starting at any $x \in X$ and successively applying $w$, one will eventually find $k, m$ such that $w^{k+m}(x) = w^k(x)+n$, and hence can conclude that $\rott(w) = n/m$.    Here is a concrete example.  

\begin{example} \label{CW ex1}
Let $p_1/q_1 = p_2/q_2 = 1/2$.  Let $X$ and $Y \subset \R$ be $\Z$-periodic sets of cardinality 2, and assume they are ordered
$ ... x_1 < y_1 < x_2 < y_2 <  x_3...$
as shown in figure \ref{per fig}.  
Let  $f$ and $g$ be $(X, p_1/q_1)$ and $(Y, p_2/q_2)$ maximal monotone maps, respectively.  
Then the orbit of $x_1$ under $w = fg$ is given by $$x_1, \, \, fg(x_1)= x_4, \,  \, (fg)^2(x_1) = x_7, \, ...$$
as depicted in figure \ref{per fig}.  Since $x_7 =  x_1+3$, we have $\rott(fg) = 3/2$.  
\end{example}

 \begin{figure*}
   \labellist 
  \footnotesize \hair 2pt
   \pinlabel $x_1$ at 15 -7
   \pinlabel $y_1$ at  58 -7
       \pinlabel $x_2$ at 92 -7 
     \pinlabel $y_2$ at 117 -7 
     \pinlabel $x_3$ at 187 -7 
     \pinlabel $...$ at 205 -7 
     \pinlabel $x_7$ at 530 -7 
     \pinlabel $g$ at 70 45 
     \pinlabel $f$ at 200 45
     \pinlabel $g$ at 340 45 
     \pinlabel $f$ at 480 45
   \endlabellist
  \centerline{
    \mbox{\includegraphics[width=5in]{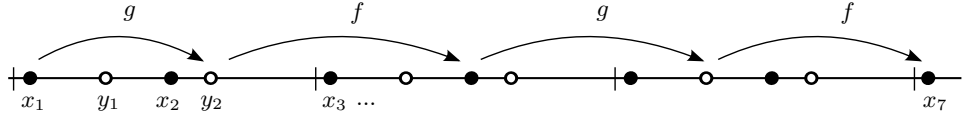}}}
 \caption{\small A periodic orbit for $fg$  (Example \ref{CW ex1})}
  \label{per fig}
  \end{figure*}

The reader should at this point convince him or herself that Example \ref{CW ex1} is not a particularly special easy case -- changing the input data or the word $w$ might result in a longer computation, but will never be more technically challenging.  
We suggest the following:

\begin{exercise} (computing rotation numbers)
\begin{enumerate}
\item With the input data of Example \ref{CW ex1}, compute $\rott(fg^2)$.  
\item Change the input data of Example \ref{CW ex1} so that the points of $X$ and $Y$ are ordered 
$$ ... x_1 < x_2 < y_1 < y_2 <  x_3 < x_4 < ...$$
and verify that the $(X, 1/2)$ and $(Y, 1/2)$ maximal monotone maps satisfy $\rott(fg) = 1$.  
\item Change the input data of Example \ref{CW ex1} again so that $p_2/q_2 = 1/3$, and $Y$ is $\Z$-periodic of cardinality 3.  (There are multiple possibilities for the ordering of points of $X$ and $Y$ in $\R$, choose one).  Compute $\rott(fg)$.     Now choose a different ordering of the points in $X$ and $Y$, and compute $\rott(fg)$ again.  
\end{enumerate}
\end{exercise}

Monotone maps can always be approximated by homeomorphisms, as stated in the next proposition. 

\begin{proposition} \label{homeo prop}
Let $f_i$ be $(X_i, p_i/q_i)$ maximal monotone maps.  
Then there exist $g_i \in \Homeo_\Z(\R)$ with $X_i$ a periodic orbit for $g_i$, and so that 
$\rott(w(g_1, ... g_n)) = \rott(w(f_1, ... f_n))$ holds for all $w$.
\end{proposition}

\begin{proof}
Let $\mathbf{X} = \bigcup_i X_i$ as before.  Pick $\epsilon$ much smaller than the minimum distance between any two points in $\mathbf{X}$, and choose homeomorphisms $g_i$ so that $g_i(x-\epsilon, x+\epsilon) \subset (f_i(x)-\epsilon, f_i(x) + \epsilon)$ holds for all $x \in \mathbf{X}$.  
By  Proposition \ref{Q prop}, there exists $x \in X$ and $m, k \in \Z$ such that
$$w(f_1, ... , f_n)^{pm}(x) = x+pk$$
for all $p$, and hence 
$$w(g_1, ... , g_n)^{pm}(x) \in (x+pk - \epsilon, x+pk + \epsilon).$$
Thus, $\rott(w(g_1, ... g_n)) = \rott(w(f_1, ... f_n)) = m/k$.
\end{proof}


Combining Propositions \ref{max prop} and \ref{homeo prop} gives an effective tool for computing the maximum possible value of $\rott(w(f_1, ... f_n))$  when $f_i$ are homeomorphisms or monotone maps with periodic orbits $X_i$.  We summarize this as an explicit algorithm.  

\begin{algorithm} \label{CW alg}
(The Calegari--Walker algorithm) \\
\parbox{.8cm}{\textit{Input:}}  
\parbox[t]{\textwidth}{ \vspace{-.5cm}
\begin{enumerate}
\item  A word $w$ in $n$ letters
\item  Rational numbers $s_i = p_i/q_i \in \Q$, 
for $i = 1, ... n$.  
\item  A choice of $p_i/q_i$-periodic sets $X_i$ of cardinality $q_i$.  
\\ (The output of the algorithm depends only on the linear ordering of $\bigcup_i X_i$, \\or equivalently, its cyclic order after projection to $S^1$).  
\end{enumerate} 
}
\textit{Process:} 
Let $f_i$ be a $(X_i, p_i/q_i)$ maximal monotone map.  Compute $\rott(w(f_1, ..., f_n))$ by finding a periodic orbit as in Proposition \ref{max prop}. \medskip
\\
\textit{Output:} 
The result is the maximum value of $\rott(w(g_1, ... g_n))$, where $g_i$ are any monotone maps (equivalently, any elements of $\Homeo_\Z(\R)$) with periodic sets $X_i$ and $\rott(g_i) = p_i/q_i$. \medskip
\\ 
\textit{Variation:}  To compute $R_w(s_1, ... s_n)$, take $p_i/q_i$ to be in lowest terms, run the algorithm over all possible configurations of linear orderings on the union of the $X_i$ and take the maximum output.  
\end{algorithm}

Note that, in the special case where $\rott(f_i) = 0$, one can effectively run the algorithm where $X_i$ is any finite subset of $\fix(f_i)$ -- this is allowed by viewing a cardinality $q$ subset of $\fix(f_i)$ as a $0/q$--periodic orbit for $f_i$.  
In \cite{Invent} one can find several worked examples computing the maximum value of $\rott(f_1 f_2 ... f_n)$ given constraints on the fixed sets of $f_i$.  

\subsection{Applications}
As an application, Algorithm \ref{CW alg} can be used to give a complete closed-form answer to Problem \ref{rot prob} for the word $fg$, and significant information in other cases.  
We start with a general lemma bounding the denominator of $R_w$. 

\begin{lemma} \label{denom lem}
Suppose $s_i$ are 
rational, that $s_1 = p/q$, and that $w$ is a word in $f_1, ... f_n$.  If 
$R_{w}(s_1, ... s_n) = n/m$ in lowest terms, then $m \leq q$.  
\end{lemma}

\begin{proof}
Apply Algorithm \ref{CW alg}.  Since $w$ ends in $f_1$, and $f_1(\mathbf{X}) \subset X_1$, a closed orbit for $w$ be a subset of $X_1$, hence is a $\Z$-periodic set of cardinality $m$, for some $m \leq q$. 
\end{proof}

\begin{remark}
As $\rott(w)$ is invariant under cyclic permutations of $w$ (this follows from conjugation-invariance), Lemma \ref{denom lem} shows that the denominator of $R_w(s_1, ... s_n)$ is bounded by $\min \{q \mid s_i = p/q \}$.  
\end{remark}

\begin{theorem}[Theorem 3.9 in \cite{CW}] \label{CW thm}   
For any $s, t \in \R$,
$$R_{fg}(s, t) = \sup \limits_{p_1/q\leq s, \,\, p_2/q \leq t} \frac{p_1 + p_2 + 1}{q}$$
where $p_i$, $q$ are integers, $q>0$.  
\end{theorem}

The proof uses a lemma, which we leave as an exercise. 

\begin{lemma}  \label{exercise lem}
For any integers $p, k$, and $q$, we have $R_{fg}(\frac{p}{q}, \frac{k}{q}) \geq \frac{p+k+1}{q}$.  In fact, if $X$ and $Y$ are $\Z$-periodic sets of cardinality $q$, ordered as 
$$ ... x_1 < y_1 < x_2 < y_2 <  x_3...$$
then the composition of the $(X, p/q)$ and $(Y, k/q)$ maximal monotone maps has rotation number $(p+k+1)/q$.  
\end{lemma}

Given this lemma, we outline the strategy of the proof of Theorem \ref{CW thm}. 
\begin{proof}[Proof outline]
By Proposition \ref{rat s prop}, it suffices to prove the formula for rational $s$ and $t$.  
Given $s= p_1/q_1$ and $t = p_2/q_2$, let $X$ and $Y$ be $\Z$-periodic sets of cardinality $q_1$ and $q_2$ respectively, with $f_{\mathrm{max}}$ and $g_{\mathrm{max}}$  the maximal $(X, p_1/q_1)$ and $(Y, p_2/q_2)$ monotone maps.  Suppose that the configuration of $X$ and $Y$ is such that $\rott(f_{\mathrm{max}} g_{\mathrm{max}})$ is maximized.  

By Lemma \ref{denom lem}, $\rott(f_{\mathrm{max}}g_{\mathrm{max}})$ is rational of the form $n/m$, (and hence there is a $n/m$--periodic orbit) for some $n\leq \min\{q_1, q_2\}$.    
Using this, Calegari and Walker give a construction for re-ordering the points of $X$ and $Y$, putting them into a standard form, without affecting this periodic orbit.   This is the technical work of the proof.   From this standard form, one can read off the estimates
$$p_1/q_1 \geq l/m \text{ and } p_2/q_2 \geq (n-l-1)/m$$ 
for some $l$.   Hence 
Hence $R_{fg}(p_1/q_1, p_2/q_2) = n/m = \frac{l + (n-l-1) + 1}{m}$, which is in the desired form.

\end{proof}

\boldhead{Commutators and Milnor--Wood.}
In the study of surface group actions, we are particularly concerned with commutators.  The Calegari--Walker algorithm does not apply directly to a commutator as it is a word involving inverses; however, Theorem \ref{CW thm} can still be used to give an answer.  

\begin{lemma}[Example 4.9 in \cite{CW}] \label{comm lemma}
Let $f, g \in \Homeo_\Z(\R)$. 
\begin{enumerate}[i)]
\item If $\rott(f) \notin \Q$ or $\rott(g) \notin \Q$, then $\rott [f,g]= 0$.
\item If $\rott(f)$ or $\rott(g)$ is of the form $p/q$, where $p/q \in \Q$ is in lowest terms (so $q>0$), then 
$$\rott [f,g]  \leq 1/q.$$
\end{enumerate}
\end{lemma}

\begin{proof}  We present an argument slightly different from that in \cite{CW}.  

For the rational case, suppose $\rott(f) = p/q$; by composing with translations (which does not affect the value of $[f,g]$) we may assume $0 \leq p/q <1$.  Then 
$$\rott(gf^{-1}g^{-1}) = -p/q;$$ 
We apply Theorem \ref{CW thm}:
If $j/m \leq p/q$ and $k/m \leq -p/q$, with strict inequality in some case, then $\frac{j + k}{m} < 0$, and since $j$ and $k$ are integers, we have 
$\frac{j + k +1}{m} \leq 0.$
If equality holds, then the maximum value is taken when $j/m = p/q$ is in lowest terms, where
$$\frac{p -p + 1}{q} = 1/q,$$
which implies that $\rott[f,g] \leq 1/q$.  

For the irrational case, suppose first that $f$ is conjugate to translation $T_\theta$.  After conjugacy we may assume $f(x) = x + \theta$, and so there exists some $x$ such that $g f^{-1} g^{-1}(x) = x- \theta$ (this is easy -- if $g f^{-1} g^{-1}(x) < x- \theta$ for all $x$, then $\rot(g f^{-1} g^{-1}) < \theta$, the same argument works if $g f^{-1} g^{-1}(x) > x- \theta$).  
Thus, $x$ is a fixed point for $[f,g]$, and $\rot[f,g] = 0$.   

In general, one can take a sequence $f_i$ of $C^2$ diffeomorphisms  that $C^0$--approximate $f$, i.e. such that $f_i(x) \to f(x)$ and $f_i^{-1}(x) \to f^{-1}(x)$ for all $x \in \R$.  By continuity of rotation number, $\rott(f_i) \to \rott(f)$, and we may take a subsequence such that $\rott(f_i)$ are either rational with denominator $q_i > i$, in which case $\rott[f_i, g] < 1/i$, or $\rott(f_i) \notin \Q$.   In this irrational case, Denjoy's theorem (which we hinted at in Section \ref{intro rig sec}, but see e.g. \cite[Ch. 3]{Navas} for a precise statement)  implies that each $f_i$ is \emph{conjugate} to an irrational rotation, so $\rott([f_i, g]) = 0$.  Continuity of $\rott$ applied to this sequence gives $\rott[f, g] = 0$.  
\end{proof}

As  a further illustration of the power of the algorithmic technique, we give a quick proof of the Milnor--Wood inequality.

\begin{theorem}(Milnor--Wood inequality, equivalent formulation) \\
Let $\rho \in \Hom(\Gamma_g, \Homeo(S^1))$.  Then 
$$\rott \left( \prod [\rho(a_i), \rho(b_i)] \right) \leq 2g-2.$$
\end{theorem}

\begin{proof}
By Lemma \ref{comm lemma}, $\rott \left( [\rho(a_i), \rho(b_i)] \right) \leq 1$.  Theorem \ref{CW thm} now gives the upper bound $\rott \left( [\rho(a_1), \rho(b_1)] [\rho(a_2), \rho(b_2)] \right) \leq 3$, and inductively,  
$$\rott \left( \prod_{i=1}^{g-1} [\rho(a_i), \rho(b_i)] \right) \leq 2g-3.$$
Since $\prod_i [\rho(a_i), \rho(b_i)]$ is an integer translation, it follows from Lemma \ref{additive lem} that 
$$\rott \left( \prod_{i=1}^{g} [\rho(a_i), \rho(b_i)] \right) = \rott \left( \prod_{i=1}^{g-1} [\rho(a_i), \rho(b_i)] \right) + \rott \left( [\rho(a_g), \rho(b_g)] \right) \leq 2g-2.$$ 
\end{proof}

\begin{remark}  \label{equality rk}
In fact, Theorem \ref{CW thm} does more than give the estimate
$$\rott \left( \prod_{i=1}^{g-1} [\rho(a_i), \rho(b_i)] \right) \leq 2g-3,$$
it implies that equality is achieved \emph{only if} $\rott \left( [\rho(a_i), \rho(b_i)]  \right)= 1$.  We'll use this fact in the next section.  
\end{remark}

\boldhead{Understanding $R_w(s,t)$.}
We conclude this section with a short remark on the interesting problem of understanding the function $R_w$ itself.  
Figure \ref{zig fig} shows the plot of $R_w(s,t)$ the word $w = fgffg$, with the $s$--axis on the left side, and $t$--axis on the right.  Because of its stairstep nautre (which all graphs of $R_w(s,t)$ share), Calegari and Walker call the graph a ``ziggurat".  

Theorem 3.11 in \cite{CW} gives a more precise description of the ``stairstep" nature of the ziggurat, as well as a faster algorithm to produce the graph.   There are many open questions -- for instance, where exactly do the jumps occur and what values are taken at these points?  (See the definition of \emph{slippery points} and the ``slippery conjecture" in \cite{CW}).   Recent progress on describing \emph{self-similarity phenomena} in ziggurats and the \emph{length} of certain steps along the ``edges" was made by A. Gordenko in \cite{Gordenko} and furthered by S. Chowdhury \cite{Chowdhury}.  

\begin{figure*}
  \centerline{
    \mbox{\includegraphics[width=3in]{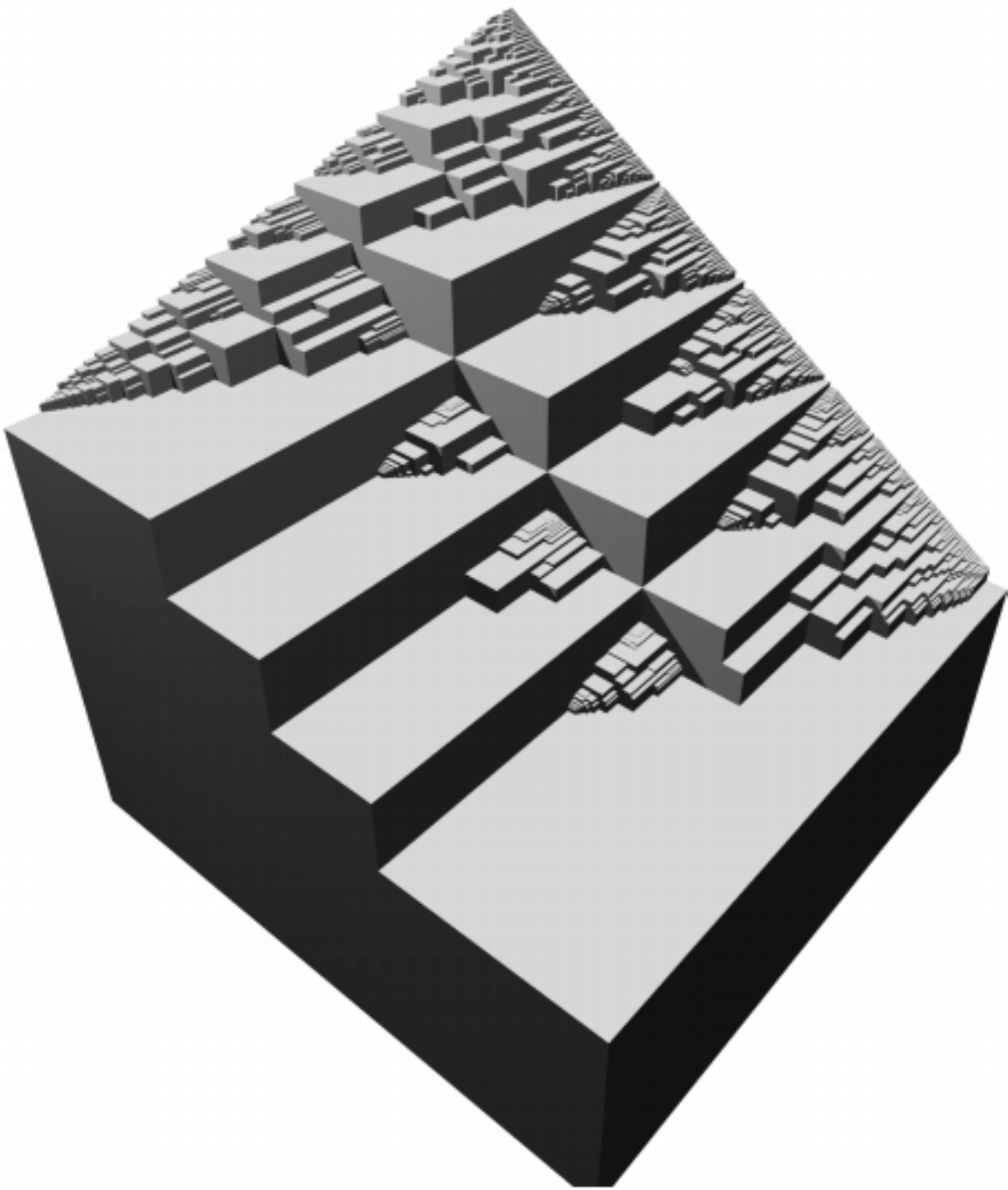}}}
 \caption{\small A 3-D plot of $R_{fgffg}(s,t)$, $0 \leq s, t < 1$, from \cite{CW}}
  \label{zig fig}
  \end{figure*}

\section{Rigidity of geometric representations}  \label{pf sec}

In this section, we explain how to use rotation number coordinates and the Calegari--Walker algorithm to prove the ``geometricity implies full rigidity" Theorem \ref{invent thm}.    

\subsection{Matsumoto's rigidity theorem}
As a warm-up and first case, we give a short proof of Corollary \ref{mat rig cor} on maximal representations -- without using Matsumoto's Theorem \ref{mats thm}.

\begin{proposition}[Corollary \ref{mat rig cor}, equivalent formulation]  \label{mat rig}
Suppose $\rho_0$ is a representation with $\rott \left( \prod [\rho_0(a_i), \rho_0(b_i)] \right) = 2g-2$.  Then $\rho_0$ is fully rigid. 
\end{proposition}

The proof, though much easier than that of Theorem \ref{invent thm}, is very much in the same spirit.   We begin with an easy lemma used in both cases.  

\begin{lemma}[Work in coordinates.]  \label{const coord lemma}
Let $\Gamma$ be any group, and suppose for each $\gamma \in \Gamma$ the function $\gamma \mapsto \rot(\rho(\gamma))$ is constant on a connected component $C$ of $\Hom(\Gamma_g, \Homeo(S^1))$.  Then $C$ consists of a single semi-conjugacy class. 
\end{lemma}

\begin{proof}
Let $\tau$ be as defined in Theorem \ref{rot coord}, and fix $\gamma$ and $\gamma'$ in $\Gamma$.  
Since $\rot(\rho(\gamma))$ and $\rot(\rho(\gamma'))$ are constant on $C$, it follows from the definition of $\tau$ that $\tau \left( \rho(\gamma), \rho(\gamma') \right)$ is constant also.   By hypothesis, the rotation numbers of generators in $\Gamma$ are also constant on $C$, so by Theorem \ref{rot coord}, $C$ is a single semi-conjugacy class. 
\end{proof}

Now we prove Proposition \ref{mat rig}. 
\begin{proof}
Let $\gamma \in \Gamma_g$.  If $\gamma$ represents a non-separating simple closed curve on $\Sigma_g$, then one can build a standard set of generators for $\Gamma_g$ with $a_1 = \gamma$.    Assume that $\rho_0$ has maximal Euler number, and 
let $\rho$ be in the connected component containing $\rho$, so $\rho$ also has maximal rotation number.  
As we mentioned in Remark \ref{equality rk}, Theorem \ref{CW thm} implies that $\rott \left( [\rho(a_1), \rho(b_1)]  \right)= 1$, which by Lemma \ref{comm lemma} implies that $\rot(\rho(a_1)) =0 = \rot(\rho_0(a_1))$.  

If $\gamma$ does not represent a non-separating simple closed curve, we may use Scott's geometric subgroup theorem from \cite{Scott} to find a finite index surface subgroup $\Lambda \subset \Gamma$ containing $\gamma$, and in which $\gamma$ represents a simple closed curve.  As Euler number and Euler characteristic are both multiplicative with respect to index, the restriction of $\rho$ to $\Lambda$ also has maximal Euler number.  Thus, we may include $\gamma$ in a standard generating set for $\Lambda$, and apply our argument above to conclude $\rot(\rho(\gamma)) = 0 = \rot(\rho_0(\gamma))$.  
\end{proof}

\begin{remark}
The proof of Proposition \ref{mat rig} consisted of showing that if $\rho$ is maximal, then $\rot(\rho(\gamma)) = 0$ for all $\gamma \in \Gamma_g$.    By virtue of Lemma \ref{const coord lemma}, this implied that the connected component of a maximal representation was a single semi-conjugacy class. 
However, the proof does \emph{not} show that all representations with Euler number $2g-2$ are semi-conjugate.   For example, the trivial representation satisfies $\rot(\rho(\gamma)) = 0$ for all $\gamma \in \Gamma_g$, but is not semi-conjugate to one with Euler number $2g-2$.   In order to get the semi-conjugacy result of Theorem \ref{mats thm}, one needs to also understand the value of $\tau(\rho(\gamma), \rho(\gamma'))$ for all pairs $\gamma$ and $\gamma'$.   This requires some careful thought; even for maximal representations the value depends on the choice of $\gamma$ and $\gamma'$.    Describing $\tau(\rho(\gamma), \rho(\gamma'))$ in terms of $\gamma$ and $\gamma'$ (topologically) is the main content of Matsumoto's proof in \cite{Matsumoto}.  
\end{remark}

\subsection{Modifications for the general case} 
Proposition \ref{mat rig} gives a special case of full rigidity-- the case of geometric representations of surface groups in to $\PSL(2,\R)$.  For the general case, we first reduce the problem to that of surface groups, then follow a similar strategy to the proof of Proposition \ref{mat rig}.    

\boldhead{I. Reduction to surface groups.}
Recall that each geometric subgroup of $\Homeo(S^1)$ is either finite cyclic or contains a finite index surface group.  If $\Gamma$ is finite cyclic, a faithful representation is determined by the rotation number of a generator, which must be of the form $k/|\Gamma|$, and hence is constant on connected components.   

The remaining case is reduced to rigidity of surface groups by the following Proposition. 

\begin{proposition}
Let $\Gamma$ be any group, $\Lambda \subset \Gamma$ a finite index subgroup, and $\rho_0: \Gamma \to \Homeo(S^1)$ a representation.  If the restriction of $\rho_0$ to $\Lambda$ is fully rigid in $\Hom(\Lambda, \Homeo(S^1))$, then $\rho_0$ is fully rigid in $\Hom(\Gamma, \Homeo(S^1))$.
\end{proposition}

\begin{proof}
Suppose that $\rho_0 |_\Lambda$ is fully rigid. 
As in Lemma \ref{const coord lemma}, it suffices to show that, for each $\gamma \in \Gamma$, the function $\gamma \mapsto \rot(\rho(\gamma))$ is constant on the connected component of $\rho_0$ in $\Hom(\Gamma, \Homeo(S^1))$.  Given $\gamma \in \Gamma$,  there exists $n$ such that $\gamma ^n \in \Lambda$.   If $\rho$ is in the same connected component as $\rho_0$, then $\rho|\Lambda$ is in the same connected component of $\Hom(\Lambda, \Homeo(S^1))$ as $\rho_0|\Lambda$, so 
$$ n \rot(\rho(\gamma)) =  \rot(\rho(\gamma^n)) =  \rot(\rho_0(\gamma^n)) = n  \rot(\rho_0(\gamma))$$
which implies that $\rot(\rho(\gamma))$ is constant on connected components.  
\end{proof}

\boldhead{II. Proof outline for surface groups}(technical work omitted).
Suppose that $\rho_0: \Gamma_g \to \PSLk$ is a geometric representation.  
Since the restriction of $\rho$ to a finite index subgroup is also geometric, it suffices (as in the proof of Proposition \ref{mat rig}) to prove that $\rot(\rho(a_1))$ is constant on the component of $\rho_0$ in $\Hom(\Gamma, \Homeo(S^1))$, where $a_1$ is any non-separating simple closed curve.  
However, unlike in Proposition \ref{mat rig}, here $\rho_0$ does not have maximal Euler number.  We will show instead that it does have maximal Euler number \emph{given} some constraints on the cyclic order of periodic points for certain elements of $\rho(\Gamma_g)$.    Rather than using the computation of $R_w(s_1,..., s_n)$ as we did in Proposition \ref{mat rig}, here we'll use the more refined output of Algorithm \ref{CW alg}, which is sensitive to the input order of periodic points.   

We begin by describing the ordering of periodic points for some elements of $\rho_0(\Gamma_g)$.   From Example \ref{pslk ex}, we know that $\rho_0$ is obtained by lifting a subgroup of $\PSL(2,\R)$ surface group action to a $k$-fold cover of $S^1$, and has Euler number $\pm (2g-2)/k$.  For concreteness, we'll work with the $(2g-2)/k$ case.

Figure \ref{fp fig} is a picture of the arrangement of fixed or periodic points of the generators $a_1$ and $b_1$, and the conjugate $b_1a_1^{-1} b_1^{-1}$ for the standard $\PSL(2,\R)$ action with Euler number $2g-2$.  Lifting this picture to a $k$-fold cover produces $2k$ fixed (or periodic) points each for $\rho_0(a_1)$ and $\rho_0(b_1a_1^{-1} b_1^{-1})$, which alternate two-by-two around the circle. 
After lifting to the line and using these configurations and rotation numbers as input data, the Calegari-Walker algorithm gives the bound 
$$\rott[\rho_0(a_1), \rho_0(b_1)] \leq 1/k,$$ 
and in this case the maximum is attained.  
The same argument can be applied to each pair of generators, giving $\rott[\rho_0(a_i), \rho_0(b_i)] = 1/k$.  

Let $c_i(\rho) = [\rho_0(a_i), \rho_0(b_i)]$, we will abuse notation and use this to refer both to the commutator in $\Homeo(S^1)$ and its canonical lift to $\Homeo_\Z(\R)$.  
Again, since we know the standard action of a surface group in $\PSL(2,\R)$, it is easy to understand the arrangement of lifted periodic points of the $c_i(\rho)$ on $\R$.  Each $c_i(\rho)$ has two $1/k$--periodic orbits, and we can choose periodic points $y_i$ and $z_i$ for $c_i(\rho)$ -- one in each orbit -- with the ordering 
\begin{equation} \label{ci order}
 y_1 < z_1 < y_2 < z_2 < ... < y_g < z_g < c_1(y_1) < c_1(z_1) < c_2(y_2) < ... 
\end{equation}
With this input data, the Calegari-Walker algorithm gives 
$$\rott(c_1 c_2.... c_{g-1}) \leq (2g-3)/k$$ 
and again equality is attained.  That equality holds can be computed directly, or argued from the fact that $\rott(c_1 c_2.... c_g) = \euler(\rho_0) = (2g-2)/k$.

   \begin{figure}
      \labellist 
  \footnotesize \hair 2pt
   \pinlabel $a_1$ at 5 42
     \pinlabel $b_1$ at -5 85 
     \pinlabel {$b_1 a_1 b_1^{-1}$} at -10 123
   \endlabellist
 \centerline{
    \mbox{\includegraphics[width=1.7in]{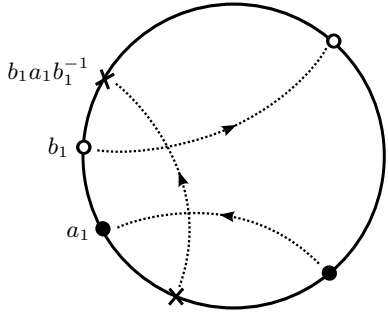}}
}
\caption{Dynamics of standard surface group elements in $\PSL(2,\R)$}
 \label{fp fig}
      \end{figure}

The upshot of our discussion is that, for any representation $\rho$ such that $\rho(a_i)$, $\rho(b_i)$ and $\rho(b_ia_i^{-1} b_i^{-1})$ have the same combinatorial configuration of fixed or periodic points as that of $\rho_0$ (call this a ``good generator configuration"), we have 
$\rott[\rho(a_i), \rho(b_i)] \leq 1/k$.    With some technical work and careful bookkeeping, involving running the Calegari--Walker algorithm on perturbations of such a representation $\rho$, one then shows that equality implies \emph{stability}: if $\rho$ has a good generator configuration, and $\rott[\rho(a_i), \rho(b_i)] = 1/k$, then there is a neighborhood of $\rho$ in $\Hom(\Gamma_g, \Homeo(S^1))$ consisting of representations that also have a good generator configuration.  
This leads to the following lemma, which is the heart of the proof.

\begin{lemma} 
Let $N \subset \Hom(\Gamma_g, \Homeo(S^1))$ be the set of representations $\rho$ of Euler number $(2g-2)/k$ such that 
\begin{enumerate}[i)]
\item For each $i$, there exist fixed or periodic points for $\rho(a_i)$, $\rho(b_i)$ and $\rho(b_i a_i^{-1} b_i^{-1})$ with the same cyclic order as in Figure \ref{fp fig}.  
\item $\rott(c_i(\rho)) = 1/k$ holds for each $i$, and  
\item there exist periodic orbits for the $c_i(\rho)$ with cyclic order as given in Equation \eqref{ci order}.  
\end{enumerate}
Then $N$ is open and closed, hence a union of connected components of $\Hom(\Gamma_g, \Homeo(S^1))$.
\end{lemma}

Since $\rho_0 \in N$, the connected component containing $\rho_0$ is a subset of $N$.  Now we may easily finish the proof by showing that $\rot(\rho(a_1))$ is constant on connected components of $N$.  Since $c_1(\rho) = [\rho(a_1), \rho(b_1)]$ has rotation number $k$, Lemma \ref{comm lemma} implies that $\rot(\rho(a_1))$ is rational with denominator at most $k$ on $N$.  As $\rot$ is continuous, $\rot(a_1)$ is constant on connected components of $N$.  
\qed

\begin{remark}[Matsumoto's alternative proof]
In a very recent paper \cite{Matsumoto mann}, Matsumoto gives a different ``coordinate-free" proof of Theorem \ref{invent thm}.   Much like how the proof we outlined here hinged on reducing rigidity to tracking some finite amount of data (i.e. configurations of fixed or periodic points for $\rho(a_i)$, $\rho(b_i)$, $\rho(b_i a_i^{-1} b_i^{-1})$ and $c_i(\rho)$), Matsumoto encodes group actions with a finite ``basic partition", similar in spirit to the Markov maps used by Bowen and Series \cite{Bowen}, \cite{BS}.    See \cite{Matsumoto mann} for details. 
\end{remark}


\subsection{Consequences of full rigidity}

As a first consequence of Theorem \ref{invent thm}, we show that the Euler number does not distinguish connected components of $\Hom(\Gamma_g, \Diff^r(S^1))$.

\begin{theorem}[Distinguishing connected components]  \label{invent comp cor}
Let $r \geq 0$, and let $k$ divide $2g-2$.  There are at least $k^{2g}$ connected components of $\Hom(\Gamma_g, \Diff^r(S^1))$ consisting of representations with Euler number $(2g-2)/k$.  
\end{theorem}

Since $k=2$ always divides $2g-2$, this gives a proof of Theorem \ref{comp thm}, which we stated in Section \ref{char var subsec}.  

\begin{proof}  
By Theorem \ref{invent thm}, it suffices to show that there are $k^{2g}$ distinct semi-conjugacy classes of geometric representations of $\Gamma_g$ with Euler number $(2g-2)/k$.   
In fact, there are exactly $k^{2g}$ geometric representations, and no two are semi-conjugate.  To see this, we return to Example \ref{pslk ex}.  Each geometric representation with Euler number $(2g-2)/k$ corresponds to a \emph{lift} $\rho'$ of a geometric representation $\rho$ to $\PSL(2,\R)$ as in the diagram below

\begin{displaymath}
    \xymatrix @M=4pt {
  0 \ar[r] & \Z/k\Z \ar[r] &   \PSLk  \ar[r] & \PSL(2,\R)  \ar[r] & 1 \\
  & &  &  \Gamma_g \ar[u]_{\textstyle \rho} \ar[ul]^{ \textstyle \rho'}  &
        }
\end{displaymath}

Since $\Hom(\Gamma_g, \Z/k\Z)$ has cardinality $k^{2g}$, there are $k^{2g}$ different lifts, corresponding to the $k$ choices of lifts for each standard generator.  As each lift of a generator has different rotation number, no two of the $k^{2g}$ choices of lifts are semi-conjugate.  
\end{proof}

The second consequence is a straightforward application of Ghys' differentiable rigidity techniques in \cite{Ghys IHES}, giving a generalization of Theorem \ref{bowden thm}.   

\begin{theorem}(Differentiable rigidity)
Let $r \geq 3$, and let $\rho_0: \Gamma \to \PSLk$ be geometric.  Then the connected component of $\rho_0$ in $\Hom(\Gamma, \Diff^r(S^1))$ consists of a finite dimensional family of $\Diff^r(S^1)$--conjugacy classes. 
\end{theorem}

\begin{proof}[Proof sketch]
By Theorem \ref{invent thm}, any representation $\rho$ in the connected component of $\rho_0$ is semi-conjugate to $\rho_0$.  The arguments from the proof of Theorem 1 in \cite{Ghys minimal} 
imply that $\rho$ and $\rho_0$ are actually conjugate -- this kind of argument is now a fairly standard technique: one considers the foliated bundle corresponding to $\rho$, and appeals to Duminy's theorem on ends of leaves of a foliation with exceptional minimal set to show that $\rho$ must be minimal, hence conjugate to $\rho_0$.  

As described in Section \ref{hyp sec}, the arguments from \cite{Ghys IHES} now imply that there exists a $C^r$ conjugacy of $\rho$ into $\PSLk$.  As the space of homomorphisms of $\Gamma_g$ into $\PSLk$ with Euler number $(2g-2)/k$ is finite dimensional (each connected component is isomorphic to the Teichmuller component of $\Hom(\Gamma_g, \PSL(2,\R))/\PSL(2,\R)$), this proves the theorem.  
\end{proof}


\section{Flexibility of group actions}  \label{flex sec}

Having covered rigidity in great detail, we now say a few words about its counterpart.   Just how rigid are non-rigid actions?    
Are there more examples -- analogous to Ghys' differentiable rigidity Theorem \ref{ghys glob rig} -- where the quotient of $\Hom(\Gamma_g, \Homeo(S^1))$ by semi-conjugacy is locally finite dimensional?   Are (as we conjectured) geometric actions the only fully rigid ones?  

In \cite{MW}, progress is made on this conjecture by studying two general families of deformations (``bending" and ``twisting"), both most easily described in terms of the topology of $\Sigma_g$.  

\begin{example}[Bending in separating curves] \label{bending}
Let $\rho: \Gamma_g \to \Homeo_0(S^1)$ be a representation, and $c \subset \Gamma_g$ an element representing a separating simple closed curve.  Then $\Gamma_g$ splits as an amalgamated free product, $A \ast_c B$, where $A$  and $B$ are the fundamental groups of the two connected components of $\Sigma_g \setminus c$.  

Let $f_t$ be a continuous family of homeomorphisms in the centralizer of $\rho(c)$, with $f_0 = \id$.  A \emph{bending deformation of $\rho$ along $f_t$} is the family of representations $\rho_t$ defined by 
$$ \rho_t(\gamma) = 
\left\{ \begin{array}{ll} 
\rho(\gamma) &\text{ for } \gamma \in A \\
f_t \rho(\gamma) f_t^{-1} &\text{ for } \gamma \in B \\
\end{array} \right.
$$
Since $f_t \rho(c) f_t^{-1} = \rho(c)$, this gives a family of well-defined homomorphisms, with $\rho_0 = \rho$.  
\end{example}

The analog of bending for a non-separating simple closed curve is a \emph{twist deformation}.  

\begin{example}[Twisting in non-separating curves]  \label{twisting}
Let $S \subset \Sigma_g$ be a genus-1 subsurface with one boundary component $c$, and let $a$ and $b$ be standard (free) generators for $A = \pi_1(S)$.  Then $[a, b] = c$, and $\Gamma_g = A\ast_c B$.      

Let $\rho: \Gamma_g \to \Homeo(S^1)$, and let $f_t$ be a one-parameter family of homeomorphisms in the centralizer of $\rho(a)$, with $f_0 = \id$ and $f_1 = f$.   
A \emph{twist deformation of $\rho$ along $a$ in $f_t$} is the family of representations defined by 
$$  \begin{array}{ll} 
\rho_t(\gamma) =  \rho(\gamma) &\text{ for } \gamma \in B \\
\rho_t(a) = \rho(a) & \\
\rho_t(b) = \rho(b) f_t &
\end{array}
$$
Since $f_t$ commutes with $\rho(a)$, we have $[\rho(a), \rho(b)f_t] = [\rho(a), \rho(b)]$, so $\rho_t$ is a well-defined representation, and $\rho_0 = \rho$. 
\end{example}

\begin{remark}  \label{twist rk}
The terminology ``bend" comes from a standard deformation of surface groups in $\PSL(2,\C)$, and ``twist" comes from the special case where $\rho(a)$ belongs to a 1-parameter family of homeomorphisms $a_t$, with $a_0 = \id$ and $a_1 = \rho(a)$.  In this case, the twist of $\rho$ by $a_1=a$ described above corresponds to performing a Dehn twist in $a$.   
\end{remark}

It would be interesting to know just how far one can get by twisting and bending.  

\begin{question}
Let $X$ be a path-component of $\Hom(\Gamma_g, \Homeo(S^1))$.  Can every two representations in $X$ be connected by a path consisting of a sequence of twist and bend deformations?  
\end{question}

Related to this is the following.

\begin{question} \label{0 bend q}
Are there two representations of $\Gamma_g$ with Euler number 0 that cannot be connected by a path of twist and bend deformations?  
\end{question}

It is an open question whether the set of representations with Euler number 0 is connected.  An answer to Question \ref{0 bend q} would provide good evidence in either direction.

\boldhead{Other groups.} 
In Section \ref{not just subsec}, we described geometrically motivated examples of actions of the fundamental groups of 3-manifolds on the circle.   Theorem \ref{fibered face thm} states that the construction for fibered hyperbolic 3-manifolds gives fully rigid actions.  Are there other fully rigid actions of such groups on $S^1$?    Bending and twist deformations don't make sense in this context, but perhaps there are other good candidates for families of deformations.

Of course, it would be very interesting to have a better picture of the ``character variety" of semi-conjugacy classes of representations of $\Gamma$ to $\Homeo(S^1)$, for any group $\Gamma$.   One approach to this might be to \emph{construct} groups $\Gamma$ whose algebraic structure makes $\Hom(\Gamma, \Homeo(S^1))$ easy to understand.  Calegari \cite{Calegari trans} has described groups with the property that some element $\gamma$ has $\rot(\rho(\gamma))$ essentially constant over all $\rho \in \Hom(\Gamma, \Homeo(S^1))$; these might be a good place to start.    Alternatively, one could follow the approach we have promoted here, taking an already interesting subgroup of $\Homeo(S^1)$ (surface group, 3-manifold group, Thompson's group,...) and study its deformations.   Which other groups might provide such examples?   

\bigskip


\bigskip
\bigskip
Dept. of Mathematics, 970 Evans Hall

University of California, Berkeley  

Berkeley, CA 94720 

\textit{E-mail}: kpmann@math.berkeley.edu


\begin{thebibliography}{99}

\bibitem{Agard} S. Agard,. 
\textit{Mostow rigidity on the line: a survey}. 
In \textit{Holomorphic Functions and Moduli II.} Springer (1988) 1-12.

\bibitem{Anosov} D. Anosov, 
\textit{Geodesic Flows on Closed Riemannian Manifolds with Negative Curvature}. 
Proc. Steklov Inst. Math., 90 (1969), 1-235.

\bibitem{Barbot} T. Barbot, 
\textit{Caract\'erisation des flots d?Anosov en dimension 3 par leurs feuilletages faibles}. 
Ergodic Theory Dynam. Systems 15 no. 2 (1995), 247-270.

\bibitem{BE} C. Bonatti, H. Eynard-Bontemps, 
\textit{Connectedness of the space of smooth actions of on the interval}. 
Ergodic Theory Dynam. Systems, preprint available on CJO2015. doi:10.1017/etds.2015.3. (2015).

\bibitem{Bowden} J. Bowden,
\textit{Contact structures, deformations and taut foliations}. 
Preprint.  arxiv:1304.3833 (2013).

\bibitem{Bowen} R. Bowen
\textit{Hausdorff dimension of quasi-circles}.
Pub. Math. IHES 50 no. 1 (1979) 11-25.

\bibitem{BS} R. Bowen, C. Series,
\textit{Markov maps associated with fuchsian groups}.
Publications MathŽmatiques de l'Institut des Hautes ƒtudes Scientifiques
Pub. Math. IHES 50 no. 1 (1979) 153-170.

\bibitem{BILW}  M. Burger, A. Iozzi, F. Labourie, A. Wienhard
\textit{Maximal Representations of Surface Groups: Symplectic Anosov Structures}. 
Pure and Appl. Math. Quarterly 1, 555-601 (2005), Special Issue: In Memory of Armand Borel.

\bibitem{BIW} M. Burger, A. Iozzi, A. Wienhard,
\textit{Surface group representations with maximal Toledo invariant}. 
Ann. of math. 172 no.1 (2010), 517-566.

\bibitem{BIW survey}  M. Burger, A. Iozzi, A. Wienhard,
\textit{Higher Teichm\"uller Spaces: from SL(2,R) to other Lie groups}.
In \textit{Handbook of Teichm\"uller theory, volume IV} IRMA Lectures in Mathematics and Theoretical Physics Vol. 19,
A. Papadopoulos, ed. (2014).  

\bibitem{BPSW} K. Burns, C. Pugh, M. Shub, A. Wilkinson, 
\textit{Recent results about stable ergodicity}.
In \textit{Smooth ergodic theory and its applications (Seattle, WA, 1999)}, Proc. Sympos. Pure Math., vol. 69, Amer. Math. Soc., Providence, RI, (2001).

\bibitem{Calegari trans} D. Calegari,
\textit{Dynamical forcing of cicrular groups}.
Trans. Amer. Math. Soc. 358 no. 8 (2006), 3473-3491

\bibitem{Calegari GT} D. Calegari, 
\textit{Universal circles for quasigeodesic flows}. 
Geom. Topol. 10 (2006), 2271-2298.

\bibitem{CD} D. Calegari, N. Dunfield,
\textit{Laminations and groups of homeomorphisms of the circle}
Invent. Math. 152 no. 1 (2003), 149-204. 


\bibitem{CW} D. Calegari, A. Walker,
\textit{Ziggurats and rotation numbers}. 
Journal of Modern Dynamics 5, no. 4 (2011), 711-746.

\bibitem{Chowdhury} S. Chowdhury,
\textit{Ziggurat fringes are self-similar}.
Preprint.  	arXiv:1503.04227 [math.GT] (2015)

\bibitem{CS} M. Culler,  P. Shalen, 
\textit{Varieties of group representations and splittings of 3-manifolds}. 
Ann. of Math. 117 (1983), 109-146.

\bibitem{EHN} D. Eisenbud, U. Hirsch, W. Neumann,
\textit{Transverse foliations of Seifert bundles and self homeomorphism of the circle}.
Comment. Math. Helvet. 56 no. 1 (1981), 638-660.

\bibitem{Fenley} S. Fenley, 
\textit{Anosov flows in 3-manifolds}. 
Ann. of Math. (2) 139 (1994), no. 1, 79-115.

\bibitem{Fisher} D. Fisher
\textit{Groups acting on manifolds: around the Zimmer program}.
In \textit{Geometry, Rigidity, and Group Actions}, Chicago Lectures in Math. 57 (2011).

\bibitem{Ghys Lef} E. Ghys,
\textit{Groupes d'hom\'eomorphismes du cercle et cohomologie born\'ee}. 
The Lefschetz centennial conference, Part III. Contemp. Math. 58 III, Amer. Math. Soc., Providence, RI, (1987), 81-106.

\bibitem{Ghys minimal} E. Ghys,
\textit{Classe d'Euler et minimal exceptionnel}. 
Topology 26 no. 1 (1987), 93-105.

\bibitem{Ghys AIF} E. Ghys,
\textit{D\'eformations de flots d'Anosov et de groupes fuchsiens}. 
Ann. Inst. Fourier (Grenoble) 42 no. 1-2 (1992), 209-247.

\bibitem{Ghys IHES} E. Ghys, 
\textit{Rigidit\'e diff\'erentiable des groupes fuchsiens}. 
Pub. Math. de l'IHES 78 (1993), 163-185.

\bibitem{Ghys invent} E. Ghys, 
\textit{Actions de r\'eseaux sur le cercle}, 
Invent. Math. 137 no. 1  (1999), 199-231.

\bibitem{Ghys Ens} E. Ghys,
\textit{Groups acting on the circle}.
L'Enseignement Math\'ematique, 47 (2001), 329-407. 

\bibitem{GS} E. Ghys, V. Sergiescu, 
\textit{Sur un groupe remarquable de diff\'eomorphismes du cercle}. 
Comment. Math. Helv. 62 (1987), 185-239.

\bibitem{Goldman thesis}  W. Goldman,
\textit{Discontinuous groups and the Euler class}. 
Doctoral thesis, UC Berkeley (1980).

\bibitem{Goldman} W. Goldman,
\textit{Topological components of spaces of representations}. 
Invent. Math. 93 no. 3 (1988), 557-607.

\bibitem{Goldman survey} W. Goldman,
\textit{Geometric structures on manifolds and varieties of representations}. 
In \textit{Geometry of Group Representations: Proc. of a Summer Research Conference July 5-11, 1987} Contemporary Math. Vol. 74. American Math. Soc. (1988), 169-198.

\bibitem{Goldman trace} W. Goldman,
\textit{The modular group action on real SL(2)--characters of a one-holed torus}.
Geometry \& Topology 7 (2003) 443-486.

\bibitem{GGKV} T. Golenishcheva-Kutuzova, A. Gorodetski, V. Kleptsyn, D. Volk, 
\textit{Tanslation Numbers Define Generators of $F_k \to \Homeo_+(S^1)$}.
Moscow Math. J. 14 no. 2 (2014), 291-308. 

\bibitem{Gordenko} A. Gordenko
\textit{Self-similarity of Jankins-Neumann ziggurat}
Preprint. 	arXiv:1503.03114 [math.DS]  (2015).

\bibitem{GP} M. Gromov, P. Pansu 
\textit{Rigidity of lattices: an introduction}. 
 In \textit{Geometric Topology: Recent Developments}, Lecture Notes in Mathematics 1504 (1991) 39-137.

\bibitem{Hasselblatt} B. Hasselblatt, 
\textit{Hyperbolic dynamical systems}. 
In \textit{Handbook of Dynamical Systems} vol. 1A, Elsevier North Holland (2002) 239-319. 

\bibitem{Iozzi} A. Iozzi
\textit{Bounded cohomology, boundary maps, and representations into $\Homeo_+(S^1)$ and $\SU(n,1)$}.
In \textit{Rigidity in Dynamics and Geometry} Springer-Verlag (2002) 237-260.

\bibitem{JN}  M. Jankins, W. Neumann, 
\textit{Rotation numbers of products of circle homeomorphisms}. 
Math. Ann. 271 no. 3 (1985), 381-400.

\bibitem{KLP} M. Kapovich, B. Leeb, J. Porti,
\textit{Morse actions of discrete groups on symmetric spaces}.
Preprint.  arXiv:1403.7671 [math.GR] (2014). 

\bibitem{KM} D. Kotschick, S. Morita,
\textit{Signatures of foliated surface bundles and the symplectomorphism groups of surfaces}.
Topology 44, no, 1 (2005), 131-149.

\bibitem{Labourie} F. Labourie,
\textit{Lectures on Representations of Surface Groups}.
Zurich Lectures in Advanced Mathematics, 17.  European Math. Soc. (2013).

\bibitem{MP} F. Malikov, R. Penner, 
\textit{The Lie algebra of homeomorphisms of the circle}. 
Adv. Math. 140 no. 2 (1998), 282-322.

\bibitem{Invent} K. Mann,
\textit{Spaces of surface group representations}.  
Invent. Math. 201 no. 2 (2015), 669-710.

\bibitem{MW} K. Mann, M. Wolff,
\textit{Geometricity implies rigidity for surface group actions}.
In preparation.  

\bibitem{Matsumoto num} S. Matsumoto,
\textit{Numerical invariants for semi-conjugacy of homeomorphisms of the circle}.
Proc. AMS 96 no.1 (1986), 163-168.

\bibitem{Matsumoto} S. Matsumoto,
\textit{Some remarks on foliated $S^1$ bundles}.
Invent. math. 90 (1987), 343--358.

\bibitem{Matsumoto mann} S. Matsumoto,
\textit{Basic partitions and combinations of group actions on the circle: A new approach to a theorem of Kathryn Mann}
Preprint.  arXiv:1412.0397 [math.DS] (2014)

\bibitem{Milnor} J. Milnor,
\textit{On the existence of a connection with curvature zero}.
Comment. Math. Helv. 32 no. 1 (1958), 215-223.

\bibitem{Mostow} G. Mostow,
\textit{Strong rigidity of locally symmetric spaces}. 
Ann. Math. Stud., vol. 78. Princeton: Princeton University Press 1973

\bibitem{Naimi} R. Naimi, 
\textit{Foliations transverse to fibers of Seifert manifolds}. 
Comment. Math. Helv. 69 no. 1(1994), 155-162.

\bibitem{Navas} A. Navas, 
\textit{Groups of circle diffeomorphisms}.
Chicago Lectures in Mathematics, University of Chicago Press, 2011. 

\bibitem{Navas comp} A. Navas,
\textit{Sur les rapprochements par conjugaison en dimension 1 et classe $C^1$}. 
Compositio Mathematica 150.07 (2014), 1183-1195.

\bibitem{RT} F. Rhodes, C. Thompson,
\textit{Rotation numbers for monotone functions on the circle}. 
J. London Math. Soc. 2.2 (1986), 360-368.

\bibitem{Scott} P. Scott,
\textit{Subgroups of surface groups are almost geometric}. 
J. London Math. Soc. (2) 17 no. 3 (1978) 555-565.

\bibitem{Shalen} P. Shalen,
\textit{Representations of 3-manifold groups}.
In \textit{Handbook of Geometric Topology},
R. Daverman and R. Sher, eds. North-Holland, Amsterdam (2002). 955--1044.

\bibitem{Sullivan} D. Sullivan,
\textit{Quasiconformal homeomorphisms and dynamics II: Structural stability implies hyperbolicity for Kleinian groups}
Acta mathematica 155 no. 1 (1985), 243-260.

\bibitem{Spatzier} R. Spatzier,
\textit{An invitation to rigidity theory}. 
In \textit{Modern dynamical systems and applications} (ed. M. Brin, B. Hasselblatt, Y. Pesin), Cambridge University Press (2004), 211-231.

\bibitem{Thurston norm} W. Thurston
\textit{A norm for the homology of 3-manifolds}. 
Memoirs of the A.M.S. 59 no.339 (1986), 99-130.

\bibitem{Toledo} D. Toledo, 
\textit{Representations of surface groups in complex hyperbolic space}. 
J. Diff. Geom., 29 no. 1 (1989), 125-133.

\bibitem{Witte} D. Witte, 
\textit{Arithmetic groups of higher Q-rank cannot act on 1-manifolds}. 
Proc. Amer. Math. Soc. 122 (1994), 333-340.

\bibitem{Wood} J. Wood,
\textit{Bundles with totally disconnected structure group}. 
Comm. Math. Helv. 51 (1971), 183-199.

\end{thebibliography}
\end{document}